\documentclass[11pt]{amsart}
\usepackage[T1]{fontenc}
\usepackage{lmodern}
\usepackage{amsmath,amsthm,amssymb,amsfonts}
\usepackage{enumerate}
\usepackage{epsfig}
\usepackage[dvipsnames,svgnames,table]{xcolor}
\usepackage{pgf,tikz}
\usetikzlibrary{calc}
\usepackage{fullpage}
\usepackage[colorlinks=true,linkcolor=RoyalBlue,urlcolor=RoyalBlue,citecolor=PineGreen]{hyperref}
\usepackage{booktabs}
\usepackage[nameinlink]{cleveref}

\newtheorem{theorem}{Theorem}[section]
\newtheorem{proposition}[theorem]{Proposition}
\newtheorem{lemma}[theorem]{Lemma}
\newtheorem{corollary}[theorem]{Corollary}
\newtheorem{conjecture}[theorem]{Conjecture}

\theoremstyle{definition}

\newtheorem{example}[theorem]{Example}

\theoremstyle{remark}
\newtheorem{remark}[theorem]{Remark}

\DeclareMathOperator{\rank}{rank}

\DeclareMathOperator{\sgn}{sgn}

\DeclareMathOperator{\cone}{cone}
\DeclareMathOperator{\tr}{tr}

\definecolor{colR}{rgb}{.932,.172,.172}
\definecolor{colB}{rgb}{.255,.41,.884}
\definecolor{colG}{rgb}{0,0.7,0}

\tikzstyle{edge}=[line width=1.5pt]
\tikzstyle{dedge}=[edge,dashed,gray]
\tikzstyle{redge}=[edge,dashed]
\tikzstyle{bedge}=[edge]
\tikzstyle{gedge}=[edge,colG]
\tikzstyle{lnode}=[circle,white,draw, fill=black,inner sep=1pt, font=\scriptsize]
\tikzstyle{hollow}=[circle,gray,draw, thick, fill=white,inner sep=0pt, minimum size=4pt]
\tikzstyle{vertex}=[fill=black,circle,inner sep=0pt, minimum size=4pt]

\title{Algebraic connectivity in normed spaces}

\author{James Cruickshank}
\email{james.cruickshank@universityofgalway.ie}
\address{School of Mathematical and Statistical Sciences,  
University of Galway, Ireland.}

\author{Sean Dewar}
\email{sean.dewar@bristol.ac.uk}
\address{School of Mathematics, University of Bristol.}

\author{Derek Kitson}
\email{derek.kitson@mic.ul.ie}
\address{Department of Mathematics and Computer Studies, Mary Immaculate College, Ireland.}

\begin{document}
\date{}

\begin{abstract}
	The algebraic connectivity of a graph $G$ in a finite dimensional real normed linear space $X$ is a geometric counterpart to the Fiedler number of the graph and can be regarded as a measure of the rigidity of the graph in $X$. We analyse the behaviour of the algebraic connectivity of $G$ in $X$ with respect to graph decomposition, vertex deletion and isometric isomorphism, and provide a general bound expressed in terms of the geometry of $X$ and the Fiedler number of the graph. Particular focus is given to the space $\ell_\infty^d$ where we present explicit formulae and calculations as well as upper and lower bounds. As a key tool, we show that the monochrome subgraphs of a complete framework in $\ell_\infty^d$ are odd-hole-free. Connections to redundant rigidity are also presented.
\end{abstract}

\maketitle

\newcommand{\jim}[1]{\textcolor{blue}{[Jim: #1]}}



\section{Introduction}\label{sec:intro}
The {\em algebraic connectivity} (or {\em Fiedler number}) of a finite simple graph $G=(V,E)$, denoted $a(G)$, is the second smallest eigenvalue of its Laplacian matrix $L(G)$. This quantity is non-negative and bounded above by the vertex connectivity of the graph. Moreover, it is positive if and only if the graph is connected.  The algebraic connectivity of graphs is well-studied and arises in numerous contexts, such as the study of isoperimetric numbers and expanders. We refer the reader to the paper of Fiedler (\cite{Fiedler73})  and to the survey articles \cite{deAbreu,mohar91} for further properties and applications. 

The {\em $d$-dimensional algebraic connectivity} of a graph, introduced by Jord\'{a}n and Tanigawa (\cite{JT22}), is a higher-dimensional analogue of algebraic connectivity. It is a non-negative number which is positive if and only if the graph is generically rigid in $d$-dimensional Euclidean space. The case $d=1$ coincides with the usual notion of algebraic connectivity.  To define the $d$-dimensional algebraic connectivity of a graph $G$ we first consider bar-joint frameworks $(G,p)$ in $\mathbb{R}^d$ obtained by assigning points $p_v$ in $\mathbb{R}^d$ to the vertices of the graph $G$. Each bar-joint framework $(G,p)$ gives rise to a  {\em  framework Laplacian matrix} $L(G,p)$ (also known as the {\em stiffness matrix}) which is positive semidefinite.  The $\binom{d+1}{2}+1$ smallest eigenvalue of $L(G,p)$ (known as the {\em rigidity eigenvalue} or {\em worst case rigidity index}) is positive if and only if the bar-joint framework $(G,p)$ is infinitesimally rigid. The $d$-dimensional algebraic connectivity of $G$ is the supremum of these rigidity eigenvalues, where the supremum is taken over all possible bar-joint frameworks $(G,p)$ in $\mathbb{R}^d$. 

In this article, we consider  framework Laplacian matrices, rigidity eigenvalues and $d$-dimensional algebraic connectivity in a broader context; replacing $d$-dimensional Euclidean space with a general finite dimensional real normed linear space $X$. 
The {\em framework Laplacian matrix} $L(G,p)$ for a framework $(G,p)$ in $X$ derives naturally from a {\em rigidity matrix} $R(G,p)$ and can be viewed as the Laplacian matrix for a matrix-weighted graph, whereby each edge of the graph is assigned a positive semidefinite $d\times d$ matrix. 
The {\em rigidity eigenvalue} for a framework in $X$ is the $k(X)+1$ smallest eigenvalue of the framework Laplacian matrix, where the value $k(X)$ is dependent on the isometry group of the normed space $X$. In many cases of interest (such as $\ell_p$ spaces with $p\not=2$) the value $k(X)$ is simply the dimension of $X$. 

The framework Laplacian matrices and rigidity eigenvalues considered here fit neatly into the cellular sheaf formalism developed in recent work of Hansen (\cite{hansen}) and offer a rich source of examples (see \Cref{s:weightedgraphs}). Although beyond the scope of this paper, there are evident connections to isoperimetric inequalities and mixing lemmas for matrix-weighted expander graphs. Indeed, interest in $d$-dimensional algebraic connectivity has been largely motivated by applications to rigidity percolation for random graphs and rigidity expanders (\cite{JT22,LNPR23,lnpr25,pmga22,pmga25}).
 The study of rigidity eigenvalues for bar-joint frameworks is interesting in its own right and arises in multi-agent formation control (\cite{zelazo,zh09}).
 The role of alternative metrics in multi-agent formation control has received some attention (e.g. \cite{bcs19,CAPR}) and so  the rigidity eigenvalues considered here may also have relevance in these application domains. 

In \Cref{s:preliminaries}, we provide some necessary background on the algebraic connectivity of graphs and on the rigidity of frameworks in normed spaces. In \Cref{s:algcon}, we introduce the notion of a framework Laplacian matrix $L(G,p)$ for a framework in a normed space $X$. We also define the algebraic connectivity of a graph $G$ in $X$, denoted $a(G,X)$, and prove several properties. Among the results, we obtain a general upper bound for $a(G,X)$ expressed in terms of the  algebraic connectivity $a(G)$ (\Cref{t:algconn}) and compute this bound for all  $\ell_p^d$ spaces with $p\not=2$ (\Cref{c:lp}).  In \Cref{s:polyhedral}, we consider the class of polyhedral normed spaces and in particular the space $\ell_\infty^d$. We first prove a structural result for the induced monochrome subgraphs of a complete framework, showing that they are necessarily odd-hole-free (\Cref{t:oddhole}). This result, which is of independent interest,  simplifies later calculations of $a(K_n,\ell_\infty^d)$. The main result is an explicit formula for $a(G,\ell_\infty^d)$ (\Cref{t:l infinity decomp}) which we use to derive upper and lower bounds and to compute the algebraic connectivity of complete graphs in $\ell_\infty^d$. In \Cref{s:rr}, we highlight some connections to  vertex-redundant rigidity and edge-redundant rigidity.

\section{Preliminaries}
\label{s:preliminaries}
 All graphs throughout are assumed to be both finite and simple.
 Given a pair of vertices $v,w\in V$ in a graph $G=(V,E)$, we write $v\sim w$ if the vertices $v$ and $w$ are adjacent in $G$. The degree of a vertex $v$ in $G$ will be denoted $\deg_G(v)$ or simply $\deg(v)$. For $n\in \mathbb{N}$,  let $[n]:=\{1,\ldots,n\}$.
The standard basis vectors for $\mathbb{R}^n$ will be denoted $b_1,\ldots,b_n$.
The orthogonal complement of a subspace $Y$ in $\mathbb{R}^n$ will be denoted $Y^\perp$. The Euclidean norm on $\mathbb{R}^n$ is denoted $\|\cdot\|_2$.
 
 The set of all $n\times n$ real matrices will be denoted $M_n(\mathbb{R})$. The Kronecker product of two matrices $A$ and $B$ is denoted $A\otimes B$.
The eigenvalues of a real symmetric matrix $A\in M_n(\mathbb{R})$ will be denoted  $\lambda_1 (A) \leq \cdots \leq \lambda_n(A)$, where each eigenvalue is repeated according to its multiplicity. The spectral norm for
an $n\times m$ matrix $A$ will be denoted $\|A\|_{2}$, 
$$\|A\|_{2} := \sup_{x\in\mathbb{R}^m,\,\|x\|_2=1} \|Ax\|_2.$$
 The following results will be required. See for example \cite[\S III]{bhatia} and \cite{ostrowski59} for further details.

\begin{theorem}[Courant-Fischer Theorem]
\label{t:courant}
Let $A$ be an $n\times n$ real symmetric matrix with linearly independent eigenvectors $y_1,\ldots,y_n\in\mathbb{R}^n$ where, for each $j\in[n]$, $y_j$ is an eigenvector for the eigenvalue  $\lambda_j(A)$.
Set $Y_0:=\{0\}$ and, for each $k\in[n]$, denote by $Y_k$ the linear span of $y_1,\ldots,y_k$ in $\mathbb{R}^n$. 
 Then, for each $j\in[n]$,
$$\lambda_j(A) = \min\,\{x^\top Ax:x\in Y_{j-1}^\perp,\,\|x\|_2=1\}.$$
\end{theorem}

\begin{theorem}[Weyl's Perturbation Theorem]
\label{t:weyl}
Let $A$ and $B$ be $n\times n$ real symmetric matrices. 
 Then, for each $j\in[n]$,
$$|\lambda_j(A)-\lambda_j(B)|\leq \|A-B\|_2.$$
\end{theorem}

\begin{theorem}[Ostrowski's Theorem]
\label{t:ostrowski}
Let $A$  be an $n\times n$ real symmetric matrix and let $S$ be an invertible $n\times n$ matrix. 
 Then, for each $j\in[n]$,
$$\lambda_j(A) = \theta_j \lambda_j(S^\top AS),$$
where $\lambda_1(S^\top S)\leq \theta_j \leq \lambda_n(S^\top S)$.
\end{theorem}

\subsection{Algebraic connectivity}
The {\em Laplacian matrix} of a graph $G=(V,E)$ is a $|V|\times|V|$ real symmetric matrix, denoted  $L(G)$, with rows and columns indexed by $V$. The $(v,w)$-entry for a pair of vertices $v,w\in V$ is, 
$$l_{v,w} := \left\{\begin{array}{cl}
\deg(v) & \mbox{if }v=w,\\
-1 & \mbox{if }v\sim w,\\
0 & \mbox{otherwise}.
\end{array}\right.$$ 
Given an orientation on the edges of $G$, denote by $s(e)$ and $r(e)$ the {\em source} and {\em range} of a directed edge $e=(s(e),r(e))$.
Following \cite{brualdi-ryser}, the {\em oriented incidence matrix} $C(G)$ is a $|E|\times |V|$ matrix with rows indexed by $E$ and columns indexed by $V$. The $(e,v)$-entry for a directed edge $e\in E$ and a vertex $v\in V$ is,
$$c_{e,v} := \left\{\begin{array}{cl}
1 & \mbox{if }s(e)=v,\\
-1 & \mbox{if }r(e)=v,\\
0 & \mbox{otherwise}.
\end{array}\right.$$
The Laplacian matrix satisfies $L(G)=C(G)^\top C(G)$ and is hence a positive semidefinite matrix. In particular,  the eigenvalues of $L(G)$ are non-negative. Note that $L(G)z=0$ where $z=[1\,\cdots\, 1]^\top$ is the all-ones vector in $\mathbb{R}^{|V|}$ and so the smallest eigenvalue of $L(G)$ is always $0$. The second smallest eigenvalue $\lambda_2(L(G))$ is called the {\em algebraic connectivity} of $G$ and is denoted $a(G)$. 

\begin{lemma}[\cite{Fiedler73}]
\label{l:fiedler}
Let $G=(V,E)$ be a  graph with vertex connectivity $v(G)$ and edge connectivity $e(G)$.
\begin{enumerate}
    \item $a(G)=0$ if and only if $v(G)=0$. 
    \item If $G$ is not a complete graph then $a(G)\leq v(G)$.
    \item If $H_1,\ldots,H_k$ are edge-disjoint spanning subgraphs of $G$ then $\sum_{i\in[k]}a(H_i)\leq a(G)$.
   \item $a(G)\geq 2e(G)(1-\cos(\pi/n))$.
\end{enumerate}
\end{lemma}

Recall that a {\em cut vertex} in a graph $G=(V,E)$ is a vertex $v\in V$ whose removal produces a disconnected graph.

\begin{lemma}[{\cite[Corollary 2.1]{Kirkland}}]\label{l:kirkland}
    Let $G=(V,E)$ be a connected graph with a cut vertex $v$.
    Then $a(G) \leq 1$, with equality if and only if $v$ is adjacent to every other vertex of $G$.
\end{lemma}

\begin{example}
\label{ex:fiedler2}
The following formulae are presented in \cite{Fiedler73}.
\begin{enumerate}
    \item $a(P_n) = 2(1-\cos(\pi/n))$ where $P_n$ is the path graph on $n$ vertices, $n\geq 2$.
    \item $a(C_n) = 2(1-\cos(2\pi/n))$ where $C_n$ is the cycle graph on $n$ vertices, $n\geq 3$.
    \item $a(K_n)=n$ where $K_n$ is the complete graph on $n$ vertices, $n\geq 2$.
\end{enumerate}
\end{example}

\subsection{Normed spaces}
Let $X=(\mathbb{R}^d,\|\cdot\|_X)$ be a $d$-dimensional real normed linear space with unit sphere $S_X=\{x\in \mathbb{R}^d:\|x\|_X=1\}$. 
A {\em support functional} for a point $x_0\in S_X$ is a linear functional $\varphi:\mathbb{R}^d\to \mathbb{R}$ such that $\varphi\left(x_0\right)=1$ and $\|\varphi\|_X^*=1$.
Here $\|\cdot\|_X^*$ denotes the dual norm,
$$\|\varphi\|_X^*:=\sup_{\|x\|_X=1}|\varphi(x)|.$$
The norm on $X$ is {\em smooth} at a point $x_0\in S_X$ if there exists exactly one support functional for $x_0$.
In this case, the unique support functional for $x_0$ is denoted $\varphi_{x_0}$ and satisfies,
$$\varphi_{x_0}(x) = \lim_{t\to0}\, \frac{1}{t}\left(\|x_0 + tx\|_X-\|x_0\|_X\right),\quad \forall\,x\in \mathbb{R}^d.$$
The support functional $\varphi_{x_0}$ will frequently be represented by its standard matrix which will be denoted by the same symbol: $$\varphi_{x_0}=[\varphi_{x_0}(b_1)\,\cdots\, \varphi_{x_0}(b_d)]\in\mathbb{R}^{1\times d},$$
where $b_1,\ldots,b_d$ is the standard basis for $\mathbb{R}^d$.

\begin{lemma}
\label{l:smooth}
    Let $X=(\mathbb{R}^d,\|\cdot\|_X)$ and $Y=(\mathbb{R}^d,\|\cdot\|_Y)$ and let $x_0$ be a smooth point in the unit sphere of $X$.
    If $\Psi:X\to Y$ is an isometric isomorphism then $y_0:=\Psi(x_0)$
    is a smooth point in the unit sphere of $Y$ and $\varphi_{x_0} = \varphi_{y_0}\circ \Psi$.
\end{lemma}

\begin{proof}
   Suppose the point $y_0$ has two support functionals $\varphi^1$ and $\varphi^2$.
Note that the compositions $\varphi^1\circ \Psi$ and $\varphi^2\circ \Psi$
 are both support functionals for the point $x_0$. 
By uniqueness, $\varphi^1\circ\Psi=\varphi^2\circ\Psi$ and so $\varphi^1=\varphi^2$. Thus, the norm on $Y$ is smooth at $y_0$. Let $\varphi_{y_0}$ be the unique support functional for $y_0$. The composition $\varphi_{y_0}\circ \Psi$ is a support functional for $x_0$ and so, again by uniqueness, $\varphi_{x_0} = \varphi_{y_0}\circ \Psi$. 
\end{proof}

A {\em rigid motion} of the normed space $X=(\mathbb{R}^d,\|\cdot\|_X)$ is a family of continuous paths, $$\alpha_x:(-1,1)\to \mathbb{R}^d,\quad x\in \mathbb{R}^d,$$ with the following properties,
\begin{enumerate}[(i)]
\item $\alpha_x(0)=x$ for all $x\in \mathbb{R}^d$, 
\item $\alpha_x(t)$ is differentiable at $t=0$ for all $x\in \mathbb{R}^d$, and,
\item $\|\alpha_x(t)-\alpha_y(t)\|_X=\|x-y\|_X$ for all $t\in(-1,1)$ and all $x,y\in \mathbb{R}^d$.
\end{enumerate}

The induced affine map $\eta:\mathbb{R}^d\to \mathbb{R}^d$, $\eta(x) = \alpha_x'(0)$, is called an {\em infinitesimal rigid motion} of the normed space $X$. 
The collection of all infinitesimal rigid motions of $X$ is a real linear space under pointwise operations, denoted $\mathcal{T}(X)$. The dimension of $\mathcal{T}(X)$ is denoted $k(X)$. 

\begin{example}
For $1\leq q<\infty$ and $d\geq 2$, let $\ell_q^d:=(\mathbb{R}^d,\|\cdot\|_q)$ denote the $d$-dimensional $\ell_q$-space with norm $\|x\|_q := \left(\sum_{i\in [d]} |x_i|^q\right)^{\frac{1}{q}}$ for each $x=(x_1,\ldots,x_d)\in \mathbb{R}^d$. 
Also, let $\ell_\infty^d=(\mathbb{R}^d,\|\cdot\|_\infty)$ where $\|x\|_\infty := \max_{i\in [d]} |x_i|$ for each $x=(x_1,\ldots,x_d)\in \mathbb{R}^d$.

\begin{enumerate}[(a)]
\item The Euclidean norm $\|\cdot\|_2$ is smooth at every point in the unit sphere of $\ell^d_2$.
The unique support functional at a point $x=(x_1,\ldots,x_d)$ in the unit sphere has standard matrix,
$$\varphi_{x} = \left[x_1\, \cdots\,\, x_d\right].$$
The space of infinitesimal rigid motions $\mathcal{T}(\ell_2^d)$ has dimension $k(\ell_2^d)=\binom{d+1}{2}$.

\item If $q\in(1,\infty)$ and $q\not=2$ then the norm $\|\cdot\|_q$ is smooth at every point in the unit sphere of $\ell^d_q$.
The unique support functional at a point $x=(x_1,\ldots,x_d)$ in the unit sphere has standard matrix,
$$\varphi_{x} = \left[\sgn(x_1)|x_1|^{q-1} \cdots\,\, \sgn(x_d)|x_d|^{q-1}\right],$$
where $\sgn$ denotes the sign function.
The space of infinitesimal rigid motions $\mathcal{T}(\ell_q^d)$ has dimension $k(\ell_q^d)=d$.

\item The norm $\|\cdot\|_1$ is smooth at points $x=(x_1,\ldots,x_d)$ in the unit sphere of $\ell^d_1$ such that $x_i\not=0$ for each $i\in[d]$.
The unique support functional at a smooth point $x=(x_1,\ldots,x_d)$ in the unit sphere has standard matrix,
$$\varphi_{x} = \left[\sgn(x_1) \,\cdots\,\, \sgn(x_d)\right].$$
%where $\sgn$ denotes the sign function.
The space of infinitesimal rigid motions $\mathcal{T}(\ell_1^d)$ has dimension $k(\ell_1^d)=d$.

\item The norm $\|\cdot\|_\infty$ is smooth at points $x=(x_1,\ldots,x_d)$ in the unit sphere of $\ell^d_\infty$ such that $|x_i|\not= |x_j|$ for all pairs $i,j\in[d]$ with $i\not=j$.
The unique support functional at a smooth point $x=(x_1,\ldots,x_d)$ in the unit sphere has standard matrix,
$$\varphi_{x} = \left[0 \,\,\cdots\,\, \overset{i}{1}\,\, \cdots \,\,0\right],$$
where $\|x\|_\infty=|x_i|.$
The space of infinitesimal rigid motions $\mathcal{T}(\ell_\infty^d)$ has dimension $k(\ell_\infty^d)=d$.
\end{enumerate}
\end{example}

\subsection{Rigidity}
A {\em  framework} in a normed space $X=(\mathbb{R}^d,\|\cdot\|_X)$ is a pair $(G,p)$ consisting of a graph $G=(V,E)$ and a point $p\in (\mathbb{R}^d)^V$, $p=(p_v)_{v\in V}$, such that for each edge $vw\in E$,
\begin{enumerate}[(i)]
\item the components $p_v$ and $p_w$ are distinct, and,
\item the norm $\|\cdot\|_X$ is smooth at the normalised vector $\frac{p_v-p_w}{\|p_v-p_w\|_X}$.
\end{enumerate}
Note that the second condition is redundant in the case of smooth norms (and in particular for the Euclidean norm). For non-smooth norms, condition $(ii)$ is a relatively mild assumption as demonstrated by the following lemma.
The set of points $p\in (\mathbb{R}^d)^V$ for which the pair $(G,p)$ is a framework in  $X=(\mathbb{R}^d,\|\cdot\|_X)$ is denoted  $\mathcal{W}(G,X)$. 
 
\begin{lemma}{\cite[Lemma 4.1]{Dewar21}}
\label{l:wpdense}
Let $X=(\mathbb{R}^d,\|\cdot\|_X)$ be a normed linear space and let $G=(V,E)$ be a graph. Then
    $\mathcal{W}(G,X)$ is a dense subset of $(\mathbb{R}^d)^V$ and is conull with respect to Lebesgue measure.
\end{lemma}

A framework  $(G,p)$ in $X=(\mathbb{R}^d,\|\cdot\|_X)$ has {\em full affine span} if the set of components $\{p_v:v\in V\}$ affinely spans $\mathbb{R}^d$.
Each infinitesimal rigid motion $\eta\in\mathcal{T}(X)$ induces a vector $u\in (\mathbb{R}^d)^V$ with components $u_v = \eta(p_v)$ for each $v\in V$. The vector $u$ is called a {\em trivial infinitesimal flex} of $(G,p)$ and the set of all such vectors is denoted $\mathcal{T}^X(G,p)$, or simply $\mathcal{T}(G,p)$.

\begin{lemma}{\cite[Lemmas 25 \& 31]{KL20}}
\label{l:trivial}
    Let $(G,p)$ be a framework in $X=(\mathbb{R}^d,\|\cdot\|_X)$ with full affine span.
    Then $\mathcal{T}(G,p)$ is a subspace of $(\mathbb{R}^d)^V$ with dimension $k(X)$.
\end{lemma}

The {\em rigidity matrix} for a framework $(G,p)$ in $X$, denoted $R(G,p)$, is an $|E| \times d|V|$ matrix with rows indexed by $E$ and columns indexed by $V\times[d]$.
The $(e,(v,i))$-entry for an edge $e\in E$ and a pair $(v,i)\in V\times[d]$ is,
\begin{equation*}
    r_{e,(v,i)} := 
    \left\{  
    \begin{array}{cl}
        \varphi_{v,w}(b_i)    &\text{ if }e=vw, \\
        0     &\text{ otherwise},
    \end{array}    \right.
\end{equation*}
where $\varphi_{v,w}^X$ denotes the unique support functional for the point $\frac{p_v-p_w}{\|p_v-p_w\|_X}$.

Every trivial infinitesimal flex of $(G,p)$ lies in the kernel of the rigidity matrix $R(G,p)$. If there are no other vectors in the kernel of $R(G,p)$ then the framework is said to be {\em infinitesimally rigid}.
Note that if $(G,p)$ has full affine span then, by \Cref{l:trivial}, $(G,p)$ is infinitesimally rigid if and only if $\rank R(G,p)=d|V|-k(X)$.
 
 Let $\mathcal{R}(G,X)$ be the set of points $p\in\mathcal{W}(G,X)$ such that the framework $(G,p)$ is infinitesimally rigid. A graph $G$ is {\em rigid} in  $X$ if $\mathcal{R}(G,X)$ is non-empty.

\begin{lemma}{\cite[Corollary 3.8]{Dewar22}}\label{l:continuity}
    Let $G=(V,E)$ be a graph with $|V| \geq d+1$ and let $X$ be a $d$-dimensional normed space.
      Then   $\mathcal{R}(G,X)$ is an open subset of $\mathcal{W}(G,X)$.
\end{lemma}

\section{Algebraic connectivity in normed spaces}
\label{s:algcon}
In this section, we introduce the framework Laplacian matrix and rigidity eigenvalue for a framework  in a general $d$-dimensional real normed linear space $X$ and establish several properties for the algebraic connectivity of a graph  in $X$. 

\subsection{Framework Laplacian matrices}
\label{s:frameworklaplacian}
Let $(G,p)$ be a framework  in a normed linear space $X=(\mathbb{R}^d,\|\cdot\|_X)$.
The {\em framework Laplacian matrix} (or {\em stiffness matrix}) $L^X(G,p)$, or simply $L(G,p)$, is the $d|V| \times d|V|$ real symmetric matrix, $$L(G,p) := R(G,p)^\top R(G,p).$$
The framework Laplacian matrix $L(G,p)$ is positive semidefinite and so the eigenvalues of $L(G,p)$ are non-negative real numbers $0\leq \lambda_1(L(G,p))\leq \lambda_2(L(G,p))\leq \cdots \leq \lambda_{d|V|}(L(G,p))$. 

\begin{lemma}
\label{l:rigidityeigenvalue}
Let $(G,p)$ be a framework with full affine span in $X=(\mathbb{R}^d,\|\cdot\|_X)$. 
\begin{enumerate}
    \item $\lambda_1(L(G,p))=\cdots = \lambda_{k(X)}(L(G,p))=0$.
    \item $(G,p)$ is infinitesimally rigid if and only if $\lambda_{k(X)+1}(L(G,p))>0$.
    \item $ \lambda_{k(X)+1}(L(G,p)) = \min\{x^\top L(G,p) x: x\in\mathcal{T}(G,p)^\perp,\,\|x\|_2=1\}$.
\end{enumerate}
\end{lemma}

\begin{proof}
$(i)$: The kernel of $L(G,p)$ contains the space $\mathcal{T}(G,p)$ of trivial infinitesimal flexes of $(G,p)$. Thus, the result follows from \Cref{l:trivial}.

    $(ii)$: Let $\lambda_i(L(G,p))$ be the smallest non-zero eigenvalue of the framework Laplacian $L(G,p)$. Then $\rank R(G,p) =\rank L(G,p) = d|V| - i+1$.
   Thus,  the framework $(G,p)$ is infinitesimally rigid  if and only if $i=k(X)+1$.

   $(iii)$: Apply the Courant-Fischer Theorem (\Cref{t:courant}) to $L(G,p)$.
\end{proof}

In light of the above lemma,  the $k(X)+1$ smallest eigenvalue $\lambda_{k(X)+1}(L(G,p))$ will be referred to as the {\em rigidity eigenvalue} for the framework $(G,p)$. 

\begin{remark}
 For the $1$-dimensional normed space $X=(\mathbb{R},|\cdot|)$, where $|\cdot|$ denotes the absolute value, note that $k(X)=1$. In this case, the rigidity matrix $R(G,p)$ for a framework $(G,p)$ in $X$ coincides with an oriented incidence matrix $C(G)$ for some orientation of the edges of $G$.
Thus, the framework Laplacian $L(G,p)$ coincides with the graph Laplacian $L(G)$ and the rigidity eigenvalue $\lambda_2(L(G,p))$ is the algebraic connectivity (or Fiedler number) of the graph $G$. In particular,   $a(G,X)=a(G)$.
\end{remark}

The framework Laplacian matrix $L(G,p)$ can be regarded as a $|V| \times |V|$ block matrix with entries in $M_d(\mathbb{R})$. The $(v,w)$-entry for a pair of vertices $v,w\in V$ is the $d \times d$ matrix, 
\begin{equation}
\label{eq:laplacian}
    L_{v,w}^p := \left\{ \arraycolsep=1.4pt\def\arraystretch{1.5}
    \begin{array}{cl}
        \sum_{x \sim v} \varphi_{v,x}^\top  \varphi_{v,x} &\text{ if } v = w, \\
        -\varphi_{v,w}^\top  \varphi_{v,w}  &\text{ if } v \neq w\text{ and } vw\in E, \\
        0_{d\times d} &\text{ otherwise}.
    \end{array}    \right.
\end{equation}
For the following result,  recall that two  matrices $A,B\in M_n(\mathbb{R})$ are {\em congruent} if there exists an invertible matrix $S\in M_n(\mathbb{R})$ such that $S^\top A S = B$.

\begin{lemma}
\label{l:congruent}
    Let $X=(\mathbb{R}^d,\|\cdot\|_X)$ and $Y=(\mathbb{R}^d,\|\cdot\|_Y)$ and let $(G,p)$ be a framework in $X$.
If $\Psi:X\to Y$ is an isometric isomorphism and $\Psi(p):=(\Psi(p_v))_{v\in V}$ then:
\begin{enumerate}[(i)]
\item The pair $(G,\Psi(p))$  is a framework in $Y$.
\item The framework Laplacian matrices $L^X(G,p)$ and $L^Y(G,\Psi(p))$ are congruent.
\end{enumerate}
\end{lemma}

\begin{proof}
$(i)$ Let $vw\in E$ be an edge in the graph $G=(V,E)$. 
Since $(G,p)$ is a framework in $X$, the points $p_v$ and $p_w$ are distinct and the unit vector $x_0:=\frac{p_v-p_w}{\|p_v-p_w\|_X}\in X$ has a unique support functional.
The map $\Psi$ is injective and so the points $\Psi(p_v)$ and $\Psi(p_w)$ are also distinct.
Let  $y_0:=
\frac{\Psi(p_v)-\Psi(p_w)}{\|\Psi(p_v)-\Psi(p_w)\|_Y}$.
Note that $y_0=\Psi(x_0)$.
Thus, by \Cref{l:smooth}, the norm on $Y$ is smooth at $y_0$. 

$(ii)$    For each edge $vw\in E$, denote by $\varphi_{v,w}^X$ and $\varphi_{v,w}^Y$ the unique support functionals for the unit vectors $\frac{p_v-p_w}{\|p_v-p_w\|_X}\in X$ and $\frac{\Psi(p_v)-\Psi(p_w)}{\|\Psi(p_v)-\Psi(p_w)\|_Y}\in Y$ respectively. Then, by \Cref{l:smooth},  $\varphi_{v,w}^X=\varphi_{v,w}^Y\circ \Psi$. 
Using \eqref{eq:laplacian}, it follows that for each pair of vertices $v,w\in V$, the $(v,w)$-entry of the respective framework Laplacian matrices satisfies, $$L_{v,w}^p = \Psi^\top L_{v,w}^{\Psi(p)}\,\Psi.$$
Hence, $L^X(G,p) = S^\top L^Y(G,\Psi(p))\, S$ where  $S:= \Psi\otimes I_n$ is invertible.
\end{proof}

\begin{lemma}\label{l:continuity2}
    Let $G=(V,E)$ be a graph with $n:=|V| \geq d+1$ and let $X=(\mathbb{R}^d,\|\cdot\|_X)$.
           The map, 
        \begin{equation*}
             \mathcal{W}(G,X) \rightarrow M_{nd}(\mathbb{R}), ~ p \mapsto L(G,p),
        \end{equation*}
        is continuous.
\end{lemma}

\begin{proof}
    The map $\mathcal{W}(G,X) \to \mathbb{R}^{|E| \times d|V|}$, $p \mapsto R(G,p)$, is continuous by \cite[Lemma 4.3]{Dewar21}.
    The result now follows as $L(G,p) = R(G,p)^\top R(G,p)$.
\end{proof}

\subsection{Algebraic connectivity in \texorpdfstring{$X$}{X}}
Let $G$ be a graph with at least $d+1$ vertices and let $X=(\mathbb{R}^d,\|\cdot\|_X)$ be a normed linear space.
The {\em algebraic connectivity of $G$ in $X$} is the value,
\begin{equation*}
    a(G,X) := \sup \left\{ \lambda_{k(X)+1}(L(G,p)) : p \in \mathcal{W}(G,X) \right\}.
\end{equation*}

\begin{proposition}
\label{p:alg}
Let $X=(\mathbb{R}^d,\|\cdot\|_X)$ be a normed linear space and let $G=(V,E)$ be a graph with at least $d+1$ vertices. 
\begin{enumerate}[(i)]
\item If $U$ is a dense subset of $\mathcal{W}(G,X)$ then,
    \begin{equation*}
    a(G,X) = \sup \left\{ \lambda_{k(X)+1}(L(G,p)) : p \in U\right\}.
    \end{equation*}
\item If $U'$ is an open and dense subset of $(\mathbb{R}^d)^V$ then,
    \begin{equation*}
    a(G,X) = \sup \left\{ \lambda_{k(X)+1}(L(G,p)) : p \in \mathcal{W}(G,X)\cap U'\right\}.
    \end{equation*}
    \end{enumerate}
\end{proposition}

\begin{proof}
    $(i)$ Let $a'(G,X):=\sup \left\{ \lambda_{k(X)+1}(L(G,p)) : p \in U\right\}$.
    Clearly, $a(G,X)\geq a'(G,X)$.
    To prove the reverse inequality holds, suppose $\lambda_{k(X)+1}(L(G,p))> a'(G,X)$ for some $p\in\mathcal{W}(G,X)$.
    The map $\mathcal{W}(G,X)\to M_{nd}(\mathbb{R})$, $p'\mapsto L(G,p')$, is continuous by \Cref{l:continuity2}.
    By Weyl's Perturbation Theorem (\Cref{t:weyl}), it follows that $\lambda_{k(X)+1}(L(G,p'))>a'(G,X)$ for all $p'\in\mathcal{W}(G,X)$ sufficiently close to $p$.
    Since $U$ is dense in $\mathcal{W}(G,X)$, there exists $p'\in U$ such that $\lambda_{k(X)+1}(L(G,p'))>a'(G,X)$.
    This is a contradiction and so $a'(G,X)\geq a(G,X)$.

    $(ii)$ By \Cref{l:wpdense}, $\mathcal{W}(G,X)$ is a dense subset of $(\mathbb{R}^d)^V$. 
    Thus, the intersection $\mathcal{W}(G,X)\cap U'$ is also  dense in $(\mathbb{R}^d)^V$. The result now follows from $(i)$.
\end{proof}

\begin{remark}
    Note that \Cref{p:alg}$(ii)$ can be applied to the set $U'$ of points $p\in (\mathbb{R}^d)^{V}$ for which the components $p_v\in\mathbb{R}^d$, $v\in V$, are distinct. In the special case where $X=\ell_2^d$, this was proved in \cite[Lemma 2.4]{LNPR23} using different methods.  
\end{remark}

Let $\mathcal{A}(G,X)$ be the set of points $p\in (\mathbb{R}^d)^V$ such that $\{p_v:v\in V\}$ has full affine span in $\mathbb{R}^d$.

\begin{lemma}
\label{l:affine}
    Let $X=(\mathbb{R}^d,\|\cdot\|_X)$ be a normed linear space and let $G=(V,E)$ be a graph.
    Then  the intersection $\mathcal{W}(G,X)\cap \mathcal{A}(G,X)$ is  dense in $(\mathbb{R}^d)^V$.
\end{lemma}

\begin{proof}
    Note that $\mathcal{A}(G,X)$ is an open and dense subset of $(\mathbb{R}^d)^V$.
    By \Cref{l:wpdense}, $\mathcal{W}(G,X)$ is a dense subset of $(\mathbb{R}^d)^V$. The result follows.
\end{proof}

\begin{proposition}
\label{p:dense}
  Let $X=(\mathbb{R}^d,\|\cdot\|_X)$ be a normed linear space and let $G=(V,E)$ be a graph with at least $d+1$ vertices. 
  If $H=(V,E(H))$ is a spanning subgraph of $G$ then,
    \begin{equation*}
    a(H,X) = \sup \left\{ \lambda_{k(X)+1}(L(H,p)) : p \in \mathcal{W}(G,X)\cap \mathcal{A}(G,X)\right\}.
\end{equation*}
In particular, 
      \begin{equation*}
    a(G,X) = \sup \left\{ \lambda_{k(X)+1}(L(G,p)) : p \in \mathcal{W}(G,X)\cap \mathcal{A}(G,X)\right\}.
\end{equation*}
\end{proposition}

\begin{proof}
     Note that $\mathcal{A}(H,X)=\mathcal{A}(G,X)$ since $H$ has the same vertex set as $G$. By \Cref{l:affine},  $\mathcal{W}(G,X)\cap \mathcal{A}(G,X)$ is dense in $(\mathbb{R}^d)^V$. 
    Since $H$ is a spanning subgraph of $G$, $\mathcal{W}(G,X)\subseteq \mathcal{W}(H,X)$. Thus, $\mathcal{W}(G,X)\cap \mathcal{A}(G,X)$ is a dense subset of $\mathcal{W}(H,X)$. The result now follows from \Cref{p:alg}$(i)$.
\end{proof}

\begin{proposition}\label{p:rigid}
   Let $X=(\mathbb{R}^d,\|\cdot\|_X)$ be a normed linear space and let $G=(V,E)$ be a graph with at least $d+1$ vertices. Then $G$ is rigid in $X$ if and only if $a(G,X) >0$.
\end{proposition}

\begin{proof}
If $G$ is rigid in $X$ then there exists $p\in \mathcal{W}(G,X)$ such that the framework $(G,p)$ is infinitesimally rigid. 
By \Cref{l:affine},  $\mathcal{W}(G,X)\cap \mathcal{A}(G,X)$ is dense in $(\mathbb{R}^d)^V$. Thus, by \Cref{l:continuity},
we may assume that $p \in \mathcal{W}(G,X)\cap\mathcal{A}(G,X)$.
By \Cref{l:rigidityeigenvalue}$(ii)$, 
$a(G,X)\geq \lambda_{k(X)+1}(L(G,p))>0$.

For the converse, suppose $a(G,X)>0$.
By \Cref{p:dense}, there exists $p\in \mathcal{W}(G,X)\cap \mathcal{A}(G,X)$ such that $\lambda_{k(X)+1}(L(G,p))>0$.
By \Cref{l:rigidityeigenvalue}$(ii)$,
$(G,p)$ is infinitesimally rigid, and so $G$ is rigid in $X$.
\end{proof}

\begin{lemma}
    Let $X=(\mathbb{R}^d,\|\cdot\|_X)$, where $d\geq 2$ and $k(X)=d$, and let $G=(V,E)$ be a graph with at least $d+1$ vertices. If $|V|<2d$ then $a(G,X)=0$. 
\end{lemma}

\begin{proof}
    Let $n:=|V|<2d$ and suppose $a(G,X)>0$.
    By  \Cref{p:dense},
    $\lambda_{d+1}(L(G,p))>0$ for some $p\in\mathcal{W}(G,X)\cap \mathcal{A}(G,X)$ and so, by \Cref{l:rigidityeigenvalue}$(i)$ and the rank-nullity theorem, 
    $$dn-d = \rank L(G,p)=\rank R(G,p) \leq |E|\leq \tfrac{n(n-1)}{2}.$$
    Thus $n\geq 2d$, a contradiction.
\end{proof}

\begin{proposition}
\label{p:isometric-eigenvalue}
Let $X=(\mathbb{R}^d,\|\cdot\|_X)$ and $Y=(\mathbb{R}^d,\|\cdot\|_Y)$, where $d\geq 1$, and let $(G,p)$ be a framework in $X$ with $n:=|V|\geq d+1$.
If $\Psi:X\to Y$ is an isometric isomorphism then, for each $j\in[dn]$,
$$\lambda_1(\Psi^\top\Psi)\lambda_j(L^X(G,p))\leq \lambda_j(L^Y(G,\Psi(p)))\leq  \lambda_n(\Psi^\top\Psi)\lambda_j(L^X(G,p)).$$
\end{proposition}

\begin{proof}
By \Cref{l:congruent}, $L^X(G,p) = (\Psi \otimes I_n)^\top L^Y(G,\Psi(p))(\Psi\otimes I_n)$.
Thus, the result follows on applying Ostrowski's Theorem (\Cref{t:ostrowski}) with $A=L^Y(G,\Psi(p))$ and $S=\Psi\otimes I_n$ and noting that $\lambda_1((\Psi^\top\Psi)\otimes I_n))=\lambda_1(\Psi^\top\Psi)$ and $\lambda_{dn}((\Psi^\top\Psi)\otimes I_n))=\lambda_n(\Psi^\top\Psi)$.
\end{proof}

\begin{corollary}
\label{c:isometric-alg}
Let $X=(\mathbb{R}^d,\|\cdot\|_X)$ and $Y=(\mathbb{R}^d,\|\cdot\|_Y)$, where $d\geq 1$, and let $G=(V,E)$ be a graph with $n:=|V|\geq d+1$.
If $\Psi:X\to Y$ is an isometric isomorphism then, for each $j\in[dn]$,
$$\lambda_1(\Psi^\top\Psi)a(G,X)\leq a(G,Y)\leq  \lambda_n(\Psi^\top\Psi)a(G,X).$$
\end{corollary}

\begin{example}
The linear map $\Psi:\ell_1^2\to\ell_\infty^2$, $\Psi(x,y)=\tfrac{1}{2}(x-y,x+y)$, is an isometric isomorphism. Note that $\Psi^\top\Psi=\tfrac{1}{2}I_2$. Thus, by \Cref{c:isometric-alg}, for any graph $G=(V,E)$ with at least $3$ vertices, 
$$a(G,\ell_\infty^2)=\tfrac{1}{2}a(G,\ell_1^2).$$
\end{example}

\subsection{Graph decompositions}
\label{s:decomp}
A {\em decomposition} of a graph $G=(V,E)$ is a collection $H_1,\ldots,H_m$ of edge-disjoint  subgraphs of $G$ such that $H_i=(V,E_i)$ for each $i\in [m]$ and $E=\cup_{i\in[m]} E_i$.

\begin{lemma}
    \label{l:union}
    Let $(G,p)$ be a framework in a normed space $X=(\mathbb{R}^d,\|\cdot\|_X)$ and let $H_1,\ldots,H_m$ be a decomposition of $G$.
    Then $L(G,p) = \sum_{i\in[m]} L(H_i,p)$.
\end{lemma}

\begin{proof}
    The statement follows readily from \eqref{eq:laplacian}.
\end{proof}

\begin{proposition}
    \label{p:union2}
    Let $X=(\mathbb{R}^d,\|\cdot\|_X)$, where $d\geq 2$, and let $G=(V,E)$ be a graph with at least $d+1$ vertices. If $H_1,\ldots,H_m$ is a decomposition of $G$ then, $$a(G,X) \geq  \max_{i\in[m]}\, a(H_i,X).$$
    Moreover, if there exists a framework $(G,p)$ in $X$ with full affine span such that $a(H_i,X)=\lambda_{k(X)+1}(L(H_i,p))$ for each $i\in[m]$ then, $$a(G,X) \geq  \sum_{i\in[m]} a(H_i,X).$$
\end{proposition}

\begin{proof}
    Let $(G,p)$ be a framework in $X$ with full affine span.
    Note that, for each $i\in[m]$, $(H_i,p)$ is also a framework in $X$ with full affine span and $\mathcal{T}(H_i,p)=\mathcal{T}(G,p)$.
    Thus, using \Cref{l:union} and \Cref{l:rigidityeigenvalue}$(iii)$, 
    \begin{eqnarray*}
        a(G,X) &\geq& \lambda_{k(X)+1}(L(G,p)) \\
        &=& \min \{x^\top L(G,p)x:x\in \mathcal{T}(G,p)^\perp, \,\|x\|_2=1\} \\
        &\geq& \sum_{i\in[m]} \min \{x^\top L(H_i,p)x:x\in \mathcal{T}(H_i,p)^\perp, \,\|x\|_2=1\}\\
        &=& \sum_{i\in[m]}\lambda_{k(X)+1}(L(H_i,p)) \qquad (\ast)\\
        &\geq& \lambda_{k(X)+1}(L(H_i,p)) \qquad \text{ for all $i \in [m]$}.
    \end{eqnarray*}
    Thus,  for each $i\in[m]$, $a(G,X)$ is an upper bound for $\{\lambda_{k(X)+1}(L(H_i,p)):p\in \mathcal{W}(G,X)\cap\mathcal{A}(G,X)\}$. 
    By \Cref{p:dense}, $a(G,X)\geq a(H_i,X)$ for each $i\in[m]$.
    The final statement follows from the penultimate step in the above calculation (labelled ($\ast$)).
\end{proof}

\begin{corollary}
\label{c:increasing}
    Let $X=(\mathbb{R}^d,\|\cdot\|_X)$, where $d\geq 2$, and let $G=(V,E)$ be a graph. If $H=(V,E(H))$ is a spanning subgraph of $G$ then,
    $$a(G,X) \geq  a(H,X).$$
\end{corollary}

\begin{proof}
    Apply \Cref{p:union2} to the decomposition $H,H^c$ where $H^c=(V,E\backslash E(H))$ is the complement of $H$ in $G$.
\end{proof}

\subsection{Operator norms}
Let $X=(\mathbb{R}^d,\|\cdot\|_X)$ be a normed space. Define,
    \begin{equation*}
        \gamma = \gamma(X) := \max \left\{ \|f\|_2^2 : f \in \mathbb{R}^d, \,\|f\|_X^* = 1 \right\}.
    \end{equation*}
Note that $\gamma(X)=\|I\|_{op}^2$ where $I$ is the identity operator $I:(\mathbb{R}^d,\|\cdot\|_X^*)\to (\mathbb{R}^d,\|\cdot\|_2)$ and $\|\cdot\|_{op}$ denotes the operator norm. 

\begin{lemma}
\label{l:gamma}
         $\gamma (\ell_p^d) = d^{\frac{2}{p}-1}$ if $1\leq p <2$ and $\gamma (\ell_p^d) = 1$ if $2\leq p\leq\infty$. 
\end{lemma}

    \begin{proof}
    To compute $\gamma(\ell^d_p)$, consider the operator norm for  the identity operator $I:(\ell^d_p)^*\to\ell^d_2$.
    Recall that $\|\cdot\|_p^*=\|\cdot\|_q$ where $\frac{1}{p}+\frac{1}{q}=1$ when $1<p<\infty$, $q=\infty$ when $p=1$ and $q=1$ when $p=\infty$.

    If $2\leq p\leq\infty$ then $1\leq q <2$ and so, $$\|I(x)\|_2 = \|x\|_2\leq \|x\|_q=\|x\|_p^*.$$
     Thus $\|I\|_{op}\leq 1$.
    To see that equality holds note that $\|I(b_i)\|_2=1=\|b_i\|_q$ for each of the standard basis vectors $b_1,\ldots,b_d$ in $\mathbb{R}^d$.
   
    If $1<p<2$ then note that, $$\|I(x)\|_2 = \|x\|_2\leq d^{\frac{1}{2}-\frac{1}{q}}\|x\|_q=d^{\frac{1}{p}-\frac{1}{2}}\|x\|_p^*.$$
    Thus $\|I\|_{op}\leq d^{\frac{1}{p}-\frac{1}{2}}$.
    Equality holds since $\|I(x)\|_2=d^{\frac{1}{2}-\frac{1}{q}}\|x\|_q$ for $x=(d^{-\frac{1}{q}},\ldots, d^{-\frac{1}{q}})$ in $\mathbb{R}^d$.

    If $p=1$ then note that, $$\|I(x)\|_2 = \|x\|_2\leq d^{\frac{1}{2}}\|x\|_\infty=d^{\frac{1}{2}}\|x\|_1^*.$$
    Thus $\|I\|_{op}\leq d^{\frac{1}{2}}$.
    Equality holds since $\|I(x)\|_2=d^{\frac{1}{2}}\|x\|_\infty$ for $x=(1,\ldots, 1)$ in $\mathbb{R}^d$.
    \end{proof}

\subsection{Vertex deletion}
It is shown below that deleting a vertex from a graph can reduce $a(G,X)$ by at most $\gamma(X)$.
Consequently, frameworks with a high value $a(G,X)$ have a high level of redundancy regarding their rigidity in $X$. See \Cref{s:rr} for related results.

Given a framework $(G,p)$ in a normed space $X$ and a subgraph $H$ of $G$, denote by $(H,p_H)$ the framework in $X$ obtained by setting $p_H=(p_v)_{v\in V(H)}$.
    
\begin{proposition}
    \label{p:vertexdeletion}
     Let $X=(\mathbb{R}^d,\|\cdot\|_X)$, where $d\geq 2$, and let $(G,p)$ be a framework in $X$ with $|V|\geq d+2$. Let $H$  be the subgraph formed from $G$ by deleting a vertex $v_0$ and all edges incident with $v_0$.
    If the framework $(H,p_H)$ has full affine span in $X$ then, 
    $$\lambda_{k(X)+1}(L(H,p_H))\geq \lambda_{k(X)+1}(L(G,p)) -\gamma(X).$$
\end{proposition}

\begin{proof}
Let $V=\{v_0, v_1,\ldots,v_{n}\}$ and suppose $v_0$ has degree $n$.
Then the framework Laplacian matrix $L(G,p)$ can be expressed as the block matrix, $$L(G,p) = \begin{bmatrix} L(H,p_H)+D & A\\
     A^\top & \sum_{i\in[n]} \varphi_{v_0,v_i}^\top  \varphi_{v_0,v_i}\end{bmatrix}$$
     where  $A\in M_{dn\times d}(\mathbb{R})$ and $D\in M_{dn}(\mathbb{R})$ are the real matrices,
     $$A =\begin{bmatrix} 
     -\varphi_{v_0,v_1}^\top  \varphi_{v_0,v_1} \\
     \vdots\\
     -\varphi_{v_0,v_n}^\top  \varphi_{v_0,v_n}
     \end{bmatrix},\qquad 
     D=\begin{bmatrix} 
      \varphi_{v_0,v_1}^\top  \varphi_{v_0,v_1} & & \\
      &\ddots&\\
      & & \varphi_{v_0,v_n}^\top  \varphi_{v_0,v_n}
     \end{bmatrix}.
     $$
     By \Cref{l:rigidityeigenvalue}$(iii)$, there exists $u\in \mathcal{T}(H,p_H)^\perp$ such that $\|u\|_2=1$ and,
     $$\lambda_{k(X)+1}(L(H,p_H)) = u^\top L(H,p_H)u.$$
     Let $x=\begin{bmatrix} u\\0\end{bmatrix}.$
     Note that $x\in \mathcal{T}(G,p)^\perp$ and $\|x\|_2=1$.
     Also,
     $$x^\top L(G,p)x = \begin{bmatrix} u^\top &0\end{bmatrix} L(G,p)\begin{bmatrix}u\\0\end{bmatrix}=
      \begin{bmatrix} u^\top &0\end{bmatrix} 
      \begin{bmatrix}(L(H,p_H)+D)u\\A^\top u\end{bmatrix}
      =u^\top (L(H,p_H)+D)u.$$
      Applying the Cauchy-Schwarz inequality, 
      $$u^\top D u 
      = \sum_{i\in[n]} \vert\varphi_{v_0,v_i}u_i\vert^2 \\
      \leq \sum_{i\in[n]} \|\varphi_{v_0,v_i}\|_2^2\|u_i\|_2^2 \\
      \leq \gamma(X)\|u\|_2^2=\gamma(X).$$
     Thus, by \Cref{l:rigidityeigenvalue}$(iii)$,
     $$\lambda_{k(X)+1}(L(G,p))\leq x^\top L(G,p)x\leq \lambda_{k(X)+1}(L(H,p_H))+\gamma(X).$$
     This concludes the proof.
\end{proof}

\begin{corollary}
    \label{c:vrr}
     Let $X=(\mathbb{R}^d,\|\cdot\|_X)$, where $d\geq 2$, and let $G=(V,E)$ be a graph with at least $d+2$ vertices. Let $H$  be the graph formed from $G$ by deleting a vertex $v_0$ and all edges incident with $v_0$.
    Then, 
    \begin{equation*}
        a(H, X) \geq a(G,X) - \gamma(X).
    \end{equation*}
\end{corollary}

\begin{proof}
Let $U'$ be the set of all points $p\in(\mathbb{R}^d)^V$ such that the pair $(H,p_H)$ is a framework in $X$ with full affine span. 
If $p\in\mathcal{W}(G,X)\cap U'$ then, by \Cref{p:vertexdeletion}, 
     $$\lambda_{k(X)+1}(L(G,p))\leq \lambda_{k(X)+1}(L(H,p_H))+\gamma(X)\leq a(H,X)+\gamma(X).$$
Note that $U'$ is open and dense in $(\mathbb{R}^d)^V$. Thus, the result follows  by \Cref{p:alg}$(ii)$.
\end{proof}

\begin{remark}
\Cref{p:vertexdeletion} includes as a special case the following bound for frameworks in $d$-dimensional Euclidean space $\ell_2^d$,  
$$\lambda_{\binom{d+1}{2}+1}(L(H,p_H))\geq \lambda_{\binom{d+1}{2}+1}(L(G,p)) -1,$$ which was proved in \cite{JT22}.
To see this recall that $k(\ell_2^d)=\binom{d+1}{2}$ and $\gamma(\ell_2^d)=1$.
\end{remark}

\subsection{Weighted graphs}
\label{s:weightedgraphs}
A {\em scalar-weighted graph} with non-negative weights is a pair $(G,\omega)$ consisting of a graph $G=(V,E)$ and a map $\omega: V\times V\to\mathbb{R}_{\geq 0}$ such that,
\begin{enumerate}[(i)]
\item $\omega(v,w)=0$ if  $vw\notin E$ (in particular, if $v=w$), and,
\item $\omega(v,w)=\omega(w,v)$ for each edge $vw\in E$.
\end{enumerate}
The {\em weighted Laplacian matrix} $L(G,\omega)$ is the $|V|\times |V|$ symmetric matrix with entries,
$$l_{v,w}^\omega := \left\{\begin{array}{cl}
\sum_{v' : v' \sim v} \omega(v,v') & \mbox{if }v=w,\\
-\omega(v,w) & \mbox{if }v\sim w,\\
0 & \mbox{otherwise}.
\end{array}\right.$$ 
As in the unweighted case, the weighted Laplacian matrix is positive semidefinite with $\lambda_1(L(G,\omega))=0$. 
Note that setting $\omega(v,w)=1$ for each edge $vw\in E$ gives  $L(G,\omega)=L(G)$.

\begin{lemma}\label{lem:increaselaplacian}
    Let $G=(V,E)$ be a graph with non-negative weights $\omega_1,\omega_2: V\times V \rightarrow \mathbb{R}_{\geq 0}$.
    If  $\omega_1(v,w) \leq \omega_2(v,w)$ for each edge $vw \in E$
    then $\lambda_2(L(G,\omega_1)) \leq \lambda_2(L(G,\omega_2))$.
\end{lemma}

\begin{proof}
    Define $\omega := \omega_2-\omega_1$.
    Then $L(G,\omega_2) = L(G,\omega_1) + L(G,\omega)$. Moreover,  each weighted Laplacian matrix $L(G,\omega_1)$, $L(G,\omega_2)$, $L(G,\omega)$ is positive semidefinite with $\lambda_1$-eigenspace containing the all-ones vector $z=[1\,\cdots\,1]^\top$.
    Let $Y$ be the linear span of $z$ and let $u\in Y^\perp$ such that $\|u\|_2=1$. Then,
    $$u^\top L(G,\omega_2) u 
       % =  u^\top (L(G,\omega_1) + L(G,\omega)) u 
    \geq  \min_{x \in Y^\perp,~ \|x\|_2 = 1} x^\top L(G,\omega_1) x\,\,+  \min_{x \in Y^\perp,~ \|x\|_2 = 1} x^\top L(G,\omega) x.$$
    Thus, by the Courant-Fischer Theorem (\Cref{t:courant}),
    $\lambda_2(L(G,\omega_2)) \geq \lambda_2(L(G,\omega_1)) +\lambda_2(L(G,\omega))$.
\end{proof}

Let $\mathcal{S}_d^+$ denote the set of positive semidefinite $d\times d$ matrices. 
Following the terminology of \cite{hansen}, a {\em matrix-weighted graph} is a pair $(G,W)$ consisting of a graph $G=(V,E)$ and a map $W: V\times V\to\mathcal{S}_d^+$ such that,
\begin{enumerate}[(i)]
\item $W(v,w)=0_{d\times d}$ if  $vw\notin E$ (in particular, if $v=w$), and,
\item $W(v,w)=W(w,v)$ for each edge $vw\in E$.
\end{enumerate}
The Laplacian matrix for the matrix-weighted graph $(G,W)$ is a positive semidefinite $d|V|\times d|V|$ matrix, denoted $L(G,W)$, with entries,
$$L_{v,w}^W := \left\{\begin{array}{cl}
\sum_{v' : v' \sim v} W(v,v') & \mbox{if }v=w,\\
-W(v,w) & \mbox{if }v\sim w,\\
0_{d\times d} & \mbox{otherwise}.
\end{array}\right.$$ 

In the following lemma, $\tr(A)$ denotes the trace of a matrix $A\in M_d(\mathbb{R})$.

\begin{lemma}{\cite[Proposition 2.2]{hansen}}
\label{l:trace}
Let $(G,W)$ be a matrix-weighted graph with weights in $\mathcal{S}_d^+$. Then,
$$\sum_{i=1}^d\lambda_{d+i}(L(G,W))\leq \lambda_2(L(G,\omega)),$$
where $(G,\omega)$ is the scalar-weighted graph with non-negative trace weighting, $$\omega:V\times V\to \mathbb{R}_{\geq0}, \quad \omega(v,w):=\tr(W(v,w)).$$    
\end{lemma}

\begin{theorem}
\label{t:algconn}
    Let $X=(\mathbb{R}^d,\|\cdot\|_X)$ be a normed space with $k(X) \leq 2d-1$ and let $G=(V,E)$ be a graph with at least $d+1$ vertices. Then,
    \begin{equation*}
         a(G,X) \leq \frac{\gamma(X)}{2d-k(X)} a(G).
    \end{equation*}
\end{theorem}

\begin{proof}
    Let $(G,p)$ be a framework in $X$. Note that the framework Laplacian matrix $L(G,p)$ is the Laplacian matrix for the  matrix-weighted graph $(G,W)$ where $W:V\times V\to\mathcal{S}_d^+$ satisfies,
    \begin{equation*}
    W(v,w) := \left\{\begin{array}{cl}
    \varphi_{v,w}^\top\varphi_{v,w}\quad &\mbox{ if }vw\in E,\\
    0_{d\times d} &\mbox{ otherwise.}
\end{array}\right.
\end{equation*}
    Let $\omega_1 : V\times V \rightarrow \mathbb{R}_{\geq 0}$ be the scalar weighting with,
\begin{equation*}
    \omega_1(v,w) := \left\{\begin{array}{cl}
    \|\varphi_{v,w}\|_2^2\quad &\mbox{ if }vw\in E,\\
    0 &\mbox{ otherwise.}
\end{array}\right.
\end{equation*}
Note that, for each edge $vw\in E$, $\omega_1(v,w) = \varphi_{v,w} \varphi_{v,w}^\top=\tr(\varphi_{v,w}^\top\varphi_{v,w})$ and so $\omega_1$ is the trace weighting associated to $(G,W)$.
Let $\omega_2: V\times V \rightarrow \mathbb{R}_{\geq 0}$  be the constant scalar weighting where,
  \begin{equation*}
    \omega_2(v,w) := \left\{\begin{array}{cl}
    \gamma(X)\quad &\mbox{ if }vw\in E,\\
    0 &\mbox{ otherwise.}
\end{array}\right.
\end{equation*}
Then, for each edge $vw\in E$, $\omega_1(v,w) = \|\varphi_{v,w}\|_2^2\leq\gamma(X) = \omega_2(v,w)$.
    By \Cref{lem:increaselaplacian} and \Cref{l:trace},
    $$ \sum_{i =1}^d \lambda_{d+i}(L(G,p)) \leq \lambda_2 (L(G, \omega_1)) \leq  \lambda_2 (L(G, \omega_2)).$$
    Thus,
    %\begin{align*}
      $$  (2d-k(X)) \lambda_{k(X)+1}(L(G,p)) 
        \leq  \sum_{i=k(X)+1}^{2d} \lambda_{i}(L(G,p)) 
        =  \sum_{i=1}^{d} \lambda_{d+i}(L(G,p)) 
      %  &\leq  \lambda_2 (L(G, \omega_1)) \\
        \leq  \lambda_2 (L(G, \omega_2)) 
        = \gamma(X) a(G).$$
    %\end{align*}
    The result follows.
\end{proof}
 
\begin{corollary}
\label{c:lp}
    Let $1 \leq p \leq \infty$ and $d\geq 2$.
    Then for any graph $G$ we have,
    $$a(G,\ell_p^d) \leq 
         \left\{ \arraycolsep=1.4pt\def\arraystretch{1.5}\begin{array}{cl}
             \frac{1}{d^{2-2/p}}a(G) &\text{ if } 1\leq p < 2,\\
             \frac{1}{d} a(G) &\text{ if } 2<p\leq \infty.
    \end{array}\right.$$
\end{corollary}

\begin{proof}
    The space $\mathcal{T}(\ell_p^d)$ of infinitesimal rigid motions has dimension $k(\ell_p^d) = d$ for all $1\leq p\leq \infty$, $p\not=2$.
    Thus, the result follows from \Cref{t:algconn} and  \Cref{l:gamma}. 
\end{proof}

\begin{remark}
    The case $p=2$ is excluded from the above corollary as the dimension $k(\ell_2^d)=\binom{d+1}{2}$ does not satisfy the required bound in the hypothesis of \Cref{t:algconn}. The bound $a(G,\ell_2^d)\leq a(G)$ was obtained in \cite{pmga22}, and in \cite{lnpr25} by different methods, for all $d\geq 2$.   
\end{remark}

\section{Algebraic connectivity in \texorpdfstring{$\ell_\infty^d$}{l-infinity}}
\label{s:polyhedral}
In this section, a formula for the algebraic connectivity of a graph in $\ell_\infty^d$ is derived (\Cref{t:l infinity decomp}) along with a variety of  upper and lower bounds.  To begin,
it is shown that the monochrome  subgraphs of a complete framework in any polyhedral normed space are odd-hole-free (\Cref{t:oddhole}). Moreover, in the $\ell_\infty$-plane, these monochrome subgraphs are perfect graphs. These results are used to calculate the algebraic connectivity of complete graphs in $\ell_\infty^d$.

\subsection{Polyhedral normed spaces}
Let $\mathcal{P}$ be a convex centrally symmetric polytope in $\mathbb{R}^d$ with facets $\pm F_1,\ldots, \pm F_m$. 
Each facet $F$ can be expressed as,
$$F=\{x\in\mathcal{P}:x\cdot\hat{F}=1\}$$
for some unique vertex $\hat{F}$ of the dual polytope $\mathcal{P}^{\Delta}$.
The conical hull of a facet $F$ will be denoted $\cone(F)$.
 The {\em polyhedral normed space} $X=(\mathbb{R}^d,\|\cdot\|_\mathcal{P})$ has norm, $$\|x\|_\mathcal{P} := \max_{j\in [m]}\, |\hat{F}_j\cdot x|,\quad \forall\,x\in \mathbb{R}^d.$$ 
 The space $\mathcal{T}(X)$ of infinitesimal rigid motions of $X$  has dimension $k(X)=d$. 
 The norm $\|\cdot\|_\mathcal{P}$ is smooth at a point $x_0$ in the unit sphere $S_X$ if and only if there exists a unique facet $F$ containing $x_0$. 
See \cite{Kitson15} for more details.

Let $(G,p)$ be a   framework  in a polyhedral normed space $X=(\mathbb{R}^d,\|\cdot\|_\mathcal{P})$.
%where $G=(V,E)$ and the polytope $\mathcal{P}$ has facets $\pm F_1,\ldots,\pm F_m$.
The induced {\em monochrome subgraphs} $G_1,\ldots,G_m$ are defined as follows: For each $j\in[m]$, $G_j$ has vertex set $V$ and edge set, $$E_j=\{vw\in E: p_v-p_w\in\cone(F_j)\cup\cone(-F_j)\}.$$ 
Note that $(G_1,\ldots,G_m)$ is a decomposition of the graph $G=(V,E)$ in the sense of \Cref{s:decomp}. The $m$-tuple $(G_1,\ldots,G_m)$ will be referred to as a {\em monochrome subgraph decomposition} of $G$ in $X$.

\begin{example}
\label{ex:K5}
Let $(K_5,p)$ be the framework in $\ell_\infty^2$ with,
\[p_{v_1} = (1,-2),\quad p_{v_2}=(-2,0),\quad p_{v_3}=(0,1),\quad p_{v_4}=(2,0),\quad p_{v_5}=(-1,2).\]
The facets of the unit sphere in $\ell_\infty^2$ are $\pm F_1$ and $\pm F_2$ where
$F_1=\{1\}\times[-1,1]$ and $F_2=[-1,1]\times\{1\}$.
Note, for example, that $p_{v_1}-p_{v_2}\in \cone(F_1)$ and so the edge $v_1v_2$ lies in the induced monochrome subgraph $G_1$. 
The framework $(K_5,p)$ together with its induced monochrome subgraphs $G_1$ and $G_2$ is illustrated in \Cref{fig:K5}.    
\end{example}

\begin{figure}
    \centering

\begin{tikzpicture}[scale=0.7]

% Coordinates of the vertices
\coordinate (v1) at (1,-2);
\coordinate (v2) at (-2,0);
\coordinate (v3) at (0,1);
\coordinate (v4) at (2,0);
\coordinate (v5) at (-1,2);

% Draw the edges of the complete graph
\draw[edge, dashed] (v1) -- (v2);
\draw[edge] (v1) -- (v3);
\draw[edge] (v1) -- (v4);
\draw[edge] (v1) -- (v5);

\draw[edge, dashed] (v2) -- (v3);
\draw[edge, dashed] (v2) -- (v4);
\draw[edge] (v2) -- (v5);

\draw[edge, dashed] (v3) -- (v4);
\draw[edge] (v3) -- (v5);

\draw[edge, dashed] (v4) -- (v5);

% Draw the vertices
\fill (v1) circle (3pt) node[right] {$v_1$};
\fill (v2) circle (3pt) node[below left] {$v_2$};
\fill (v3) circle (3pt) node[above right] {$v_3$};
%\fill (v3) circle (3pt) node[above right] {$v_3$};
\fill (v4) circle (3pt) node[above right] {$v_4$};
\fill (v5) circle (3pt) node[above left] {$v_5$};

\end{tikzpicture}

    \caption{The framework $(K_5,p)$ in \Cref{ex:K5} together with the induced monochrome subgraphs $G_1$ (dashed) and $G_2$ (solid).}
    \label{fig:K5}
\end{figure}
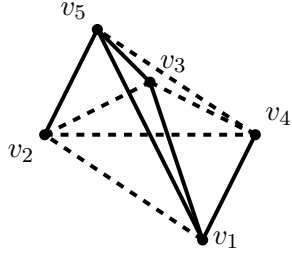 

Each path $P=(v_1,\ldots,v_k)$ in a monochrome subgraph $G_j$ has an induced edge labelling $\lambda_P$ whereby, for each $i\in[k-1]$,
$$\lambda_P(v_iv_{i+1})=
\left\{\begin{array}{cl}
1 & \mbox{if }p_{v_i}-p_{v_{i+1}}\in\cone(F_j), \\
-1 & \mbox{if }p_{v_i}-p_{v_{i+1}}\in\cone(-F_j).
\end{array}\right.$$
Denote by $P^+$ (respectively, $P^-$) the subgraph of $G_j$ with vertex set $v_1,\ldots,v_k$ and edge set $\lambda_P^{-1}(1)$ (respectively, $\lambda_P^{-1}(-1)$).
The {\em cluster graph} induced by $P^+$ (respectively, $P^-$) is the graph obtained by adding edges to each connected component of $P^+$ (respectively, $P^-$) so that each connected component is a clique. 

\begin{lemma}\label{l:path2}
	Let $(K_n,p)$ be a framework in a polyhedral normed space $(\mathbb{R}^d,\|\cdot\|_\mathcal{P})$
    and let  $G_j$ be a monochrome subgraph. 
    If $G_j$ contains a  path $P$
	then $G_j$ contains the cluster graphs induced by $P^+$ and $P^-$.
\end{lemma}

\begin{proof}
Let $P=(v_1,\ldots,v_k)$ and suppose $\lambda_P(v_{i-1}v_i)=\lambda_P(v_iv_{i+1})$ for two adjacent edges $v_{i-1}v_i$ and $v_iv_{i+1}$ in the path $P$.
   Then, 
   $$p_{v_{i-1}}-p_{v_{i+1}}=(p_{v_{i-1}}-p_{v_{i}})+(p_{v_{i}}-p_{v_{i+1}})  \in\cone(F).$$ 
Thus the edge $v_{i-1}v_{i+1}$ lies in $G_j$. It follows that each connected component of $P^+$ (and similarly of $P^-$) spans a clique which lies in $G_j$. 
\end{proof}

A {\em hole} in a graph $G$ is a vertex-induced subgraph which is a cycle of length four or more. A hole in $G$ is {\em odd} if it is a cycle of odd length. A graph $G$ is {\em odd-hole-free} if no vertex-induced subgraph of $G$ is an odd hole. 

\begin{theorem}
\label{t:oddhole}
    Let $(K_n,p)$ be a framework in a polyhedral normed space $X=(\mathbb{R}^d,\|\cdot\|_\mathcal{P})$ with induced monochrome subgraphs $G_1,\ldots,G_m$. 
    Then  $G_1,\ldots,G_m$ are odd-hole-free graphs.
\end{theorem}

\begin{proof}
Suppose $G_j$ contains an odd hole $H$ of length $k\geq 5$.
Then $H$ contains a path $P=(v_1,\ldots,v_k)$ which has an induced edge-labelling $\lambda_P$.
By \Cref{l:path2}, $G_j$ contains the cluster graphs induced by $P^+$ and $P^-$.
However, since $H$ is a vertex-induced cycle in $G_j$, the connected components of $P^+$ and $P^-$ cannot contain more than one edge. It follows that the edge-labelling $\lambda_P$ is a proper $2$-edge colouring of $P$.
In particular, since $k$ is odd, $\lambda_P(v_1v_2)\not=\lambda_P(v_{k-1}v_k)$. 
Without loss of generality, assume $\lambda_{P}(v_1v_2)=1$ and 
$\lambda_{P}(v_{k-1}v_k)=-1$.

Consider the path $Q=(v_{k-1},v_k,v_1,v_2)$ in $G_j$ together with its induced edge-labelling $\lambda_{Q}$.
By \Cref{l:path2}, $G_j$ contains the cluster graphs induced by $Q^+$ and $Q^-$.
Note that $\lambda_{Q}(v_1v_2)=\lambda_{P}(v_1v_2)=1$ and
$\lambda_{Q}(v_{k-1}v_k)=\lambda_{P}(v_{k-1}v_k)=-1$.
If $\lambda_{Q}(v_kv_1)=1$ then the cluster graph induced by $Q^+$ contains the edge $v_2v_k$.
If $\lambda_{Q}(v_kv_1)=-1$ then the cluster graph induced by $Q^-$ contains the edge $v_1v_{k-1}$.
In either case there is a contradiction since $H$ is a vertex-induced cycle in $G_j$.
\end{proof}

An {\em antihole} in a graph $G$ is a vertex-induced subgraph of $G$ that is the graph complement of a hole. An antihole is {\em odd} if it is the complement of an odd hole.  A graph $G$ is {\em odd-antihole-free} if no vertex-induced subgraph of $G$ is an odd antihole. 

\begin{theorem}
\label{t:perfect}
    Let $(K_n,p)$ be a framework in a polyhedral normed space $(\mathbb{R}^2,\|\cdot\|_\mathcal{P})$ where the polygon $\mathcal{P}$ is a quadrilateral. Let $G_1$ and $G_2$ be the induced monochrome subgraphs. 
    Then,
    \begin{enumerate}
    \item $G_1$ and $G_2$ are odd-antihole-free graphs.
    \item $G_1$ and $G_2$ are perfect graphs.
    \end{enumerate}
\end{theorem}

\begin{proof}

$(i)$:
Suppose $G_1$ contains an odd antihole. Then its complement $G_2$ contains an odd hole, which contradicts \Cref{t:oddhole}$(ii)$. 

$(ii)$:
By the Strong Perfect Graph Theorem (\cite{CRST06}), a graph $G$ is perfect if and only if it is both odd-hole free and odd-antihole-free. Thus, the result follows from $(i)$ and \Cref{t:oddhole}.
\end{proof}

\subsection{Algebraic connectivity in \texorpdfstring{$\ell_\infty^d$}{l-infinity}}
We now focus on the specific polyhedral normed space $\ell_\infty^d=(\mathbb{R}^d,\|\cdot\|_\infty)$ where $\|x\|_\infty := \max_{i\in [d]} |x_i|$ for each $x=(x_1,\ldots,x_d)\in \mathbb{R}^d$. 

For $i\in[d]$, denote by $B_i:=b_ib_i^\top$ the $d\times d$ matrix unit with $(i,i)$-entry $1$ and zero entries elsewhere.  
Recall that two matrices $A,B\in M_n(\mathbb{R})$ are {\em similar} if there exists an invertible matrix $S\in M_n(\mathbb{R})$ such that $B=S^\top AS$.

\begin{lemma}
\label{l:linfinitylaplacian}
Let $(G,p)$ be a framework in $\ell_\infty^d$ with monochrome subgraph decomposition $(G_1,\ldots,G_d)$.
\begin{enumerate}[(i)]
\item For each monochrome subgraph $G_i$,
$L(G_i,p) = L(G_i)\otimes B_i$.
\item $L(G,p)$ is similar to the block diagonal matrix $ \bigoplus_{i\in[d]} L(G_i)$. 
\item $\lambda_{d+1}(L(G,p)) = \min_{i\in [d]}\,\lambda_{2}(L(G_i))$.
\end{enumerate}
\end{lemma}

\proof
$(i)$: For each edge $vw\in E_i$, the support functional $\varphi_{v,w}$ has standard matrix $\pm b_i^\top\in\mathbb{R}^{1\times d}$.
Thus, using \eqref{eq:laplacian}, the framework Laplacian $L(G_i,p)$ is the block matrix with entries,
\begin{equation*}
    L_{v,w}^p = \left\{ \arraycolsep=1.4pt\def\arraystretch{1.5}
    \begin{array}{cl}
        \deg(v)B_i &\text{ if } v = w, \\
        -B_i  &\text{ if } v \neq w\text{ and } vw\in E, \\
        0_{d\times d} &\text{ otherwise}.
    \end{array}    \right.
\end{equation*}

$(ii)$: By \Cref{l:union} and $(i)$,
  $$L(G,p) =\sum_{i\in[d]} L(G_i,p)= \sum_{i\in[d]}L(G_i)\otimes B_i =\sum_{i\in[d]}P(B_i\otimes L(G_i))P^\top =P\left(\bigoplus_{i\in[d]} L(G_i)\right) P^\top ,$$
  where $P$ is the $d|V|\times d|V|$ ``perfect shuffle'' permutation matrix.

  $(iii)$: By $(ii)$, the framework Laplacian matrix $L(G,p)$ and the direct sum $\bigoplus_{i\in[d]} L(G_i)$ are similar and so  have the same set of eigenvalues (including multiplicities). 
  The set of eigenvalues of $\bigoplus_{i\in[d]} L(G_i)$ is the union of the eigenvalues of the Laplacian matrices $L(G_1),\ldots,L(G_d)$ (again counting multiplicities).
  Note that $\lambda_1(L(G_i))=0$ for each $i\in[d]$ and so the result follows. 
\endproof

Let $G=(V,E)$ be a graph and fix $d\geq 2$. 
Denote by  $\mathcal{M}=\mathcal{M}(G,\ell_\infty^d)$ the set of all monochrome subgraph decompositions $(G_1,\ldots,G_d)$ of $G$ in $\ell_\infty^d$.
    
\begin{theorem}\label{t:l infinity decomp}
    Let $G=(V,E)$ be a graph with at least $d+1$ vertices where $d\geq 2$.
    Then,
    \begin{equation*}
        a(G,\ell_\infty^d) = \max_{(G_1,\ldots,G_d) \in \mathcal{M}}\, \min_{i \in [d]}\, a(G_i).
    \end{equation*}
\end{theorem}

\begin{proof}
  Let $(G_1,\ldots,G_d)\in \mathcal{M}(G,\ell_\infty^d)$ be a  monochrome subgraph decomposition induced by a framework $(G,p)$ in $\ell_\infty^d$.
  Recall that $k(\ell_\infty^d)=d$. Thus, by \Cref{l:linfinitylaplacian}, $$a(G,\ell_\infty^d) \geq\lambda_{d+1}(L(G,p)) = \min_{i\in [d]}\,a(G_i).$$
There are at most finitely many framework Laplacian matrices $L(G,p)$ that can be constructed from the set of points $p\in \mathcal{W}(G,\ell_\infty^d)$.
Thus, $a(G,\ell_\infty^d) = \lambda_{d+1}(L(G,p'))$ for some $p'\in\mathcal{W}(G,\ell_\infty^d)$.
In  particular, by Lemma \ref{l:linfinitylaplacian}, $a(G,\ell_\infty^d) = \min_{i\in [d]}\,a(G_i')$ where $(G_1',\ldots,G_d')\in \mathcal{M}(G,\ell_\infty^d)$ is the monochrome subgraph decomposition induced by the framework $(G,p')$.
  \end{proof}

For the following corollary,  recall that the \emph{Cartesian product} (or \emph{box product}) of graphs $G_1=(V_1,E_1)$ and $G_2=(V_2,E_2)$ is the graph $G_1 ~ \square ~ G_2 := (V_1 \times V_2, E_1 ~ \square ~ E_2)$ where,
\begin{equation*}
    \{(v_1, v_2),(w_1,w_2)\} \in E_1 ~ \square ~ E_2  \quad \iff \quad v_1 = w_1 \text{ and } v_2w_2 \in E_2, \text{ or, } v_2 = w_2 \text{ and } v_1w_1 \in E_1.
\end{equation*}

\begin{corollary}\label{c:linfbound}
    Let $G=(V,E)$ be a graph with $n:=|V| \geq d+1$ vertices where $d\geq 2$.
    \begin{enumerate}
    \item $a(G,\ell_\infty^d) \leq a(G,\ell_\infty^{d-1})$.
    \item $a(G,\ell_\infty^d) \leq a(G)/d$.
    \item Either $a(G,\ell_\infty^d)=0$ or $a(G,\ell_\infty^d) \geq 2(1-\cos (\pi/n))$.
    \item $a(G,\ell_\infty^2)  = \max_{(G_1,G_2) \in \mathcal{M}} \,a(G_1 ~ \square ~ G_2)$.
    \end{enumerate}
\end{corollary}

\begin{proof}
    $(i)$: Let $\pi: (\mathbb{R}^d)^V \rightarrow (\mathbb{R}^{d-1})^V$ be the map that projects every component $p_v\in \mathbb{R}^d$ of a point $p=(p_v)_{v\in V}\in (\mathbb{R}^d)^V$ onto its first $d-1$ coordinates.
    Choose any framework $(G,p)$ in $\ell_\infty^d$ such that $\lambda_{d+1}(L(G,p))$ is maximal and $(G,\pi(p))$ is a framework in $\ell_\infty^{d-1}$;
    this is possible since $\mathcal{W}(G,\ell_\infty^d)$ is an open and dense subset of $(\mathbb{R}^d)^V$ (\Cref{l:wpdense}) and the function $x\mapsto L(G, x)$ is locally constant on $\mathcal{W}(G,\ell_\infty^d)$.
    If $G_1,\ldots, G_d$ are the monochrome subgraphs of $(G,p)$ and $G_1',\ldots, G'_{d-1}$ are the monochrome subgraphs of $(G,\pi(p))$,
    then $G_i \subseteq G_i'$ for each $i \in [d-1]$.
    By \Cref{t:l infinity decomp} and \Cref{l:fiedler}$(iii)$,
    $$a(G,\ell_\infty^d) = \lambda_{d+1}(L(G,p)) = \min_{i \in [d]} a(G_i) \leq \min_{i \in [d-1]} a(G'_i) \leq a(G,\ell_\infty^{d-1}).$$
    
    $(ii)$: By \Cref{t:l infinity decomp}, there exists a monochrome subgraph decomposition $(G_1,\ldots, G_d) \in \mathcal{M}(G,\ell_\infty^d)$ such that $a(G,\ell_\infty^d) = \min_{i \in [d]}\, a(G_i)$.
    %$a(G_1) \leq \ldots \leq a(G_d)$ and 
    Thus,  
    $$
        d \cdot a(G,\ell_\infty^d) = d \cdot \min_{i \in [d]}\, a(G_i) \leq \sum_{i=1}^d a(G_i)\leq a(G),
   $$
    where the final inequality follows from \Cref{l:fiedler}$(iii)$.

    $(iii)$: Suppose $a(G,\ell_\infty^d)>0$. 
    By \Cref{t:l infinity decomp}, there exists a monochrome subgraph decomposition $(G_1,\ldots,G_d)\in \mathcal{M}(G,\ell_\infty^d)$ such that $a(G,\ell_\infty^d) =\min_{i\in[d]} \,a(G_i)$.
    By \Cref{l:fiedler}$(i)$,  
    the monochrome subgraphs $G_1,\ldots,G_d$ are connected spanning subgraphs of $G$.
Thus, by \Cref{l:fiedler}$(iv)$, 
$$a(G,\ell_\infty^d)   =\min_{i\in[d]} \,a(G_i)
\geq 2(1-\cos (\pi/n)).$$

$(iv)$:  By \Cref{t:l infinity decomp}, $a(G,\ell_\infty^2) =\min\,\{ a(G_1),a(G_2)\}$ for some monochrome subgraph decomposition $(G_1,G_2)\in \mathcal{M}(G,\ell_\infty^2)$. By \cite[Theorem 3.4]{Fiedler73}, $\min\,\{ a(G_1),a(G_2)\}=a(G_1 ~ \square ~ G_2)$.
\end{proof}

\subsection{An upper bound for \texorpdfstring{$a(G,\ell_\infty^d)$}{a(G,l-infinity)}}

Let $z=[1\,\cdots\,1]^\top \in\mathbb{R}^n$ and define
$z_i=b_i\otimes z\in\mathbb{R}^{nd}$ for each $i\in[d]$.
Let $Z$ be the subspace of $\mathbb{R}^{nd}$ spanned by the orthogonal vectors $z_1,\ldots, z_d$.

\begin{lemma}
\label{l:upperbound}
Let $M=(m_{ij})$ be a symmetric positive semidefinite $nd\times nd$ matrix such that $M(Z)$=0. Then,
$$\lambda_{d+1}(M) \leq \frac{n}{n-1}\,\min_{i\in[dn]} m_{ii}.$$
\end{lemma}

\begin{proof}
By the Courant-Fischer Theorem (\Cref{t:courant}),
$$\lambda_{d+1}(M)=\min\,\{x^\top Mx: \, x\in Z^\perp,\, \|x\|_2=1\}.$$
Let $J$ be the $n\times n$ matrix with all entries equal to $1$.
Let $\tilde{M} = M-\lambda_{d+1}(M)(I_{dn}-\frac{1}{n}I_d\otimes J)$.
Note that $z^\top \tilde{M}z=0$ for all $z\in Z$.
Also, for each $x\in Z^\perp$ with $\|x\|_2=1$,
$$x^\top \tilde{M}x = x^\top Mx-\lambda_{d+1}(M)\geq 0.$$
Thus $\tilde{M}$ is positive semidefinite.
This in turn implies the diagonal entries of $\tilde{M}$ are non-negative,
and so,
$$\min_{i\in[dn]}\, m_{ii}- \lambda_{d+1}(M)\left(1-\frac{1}{n} \right)\geq 0.$$
The result now follows.
\end{proof}

\begin{theorem}\label{thm:upperbound}
    Let $G=(V,E)$ be a graph with $n$ vertices, where $n\geq d+1$, and let $d\geq 1$.
    Then, $$a(G,\ell_\infty^d) \leq \frac{n}{n-1}  \biggl\lfloor\frac{1}{d}\,\min_{v\in V}\, \deg_{G}(v)\biggr\rfloor.$$ 
\end{theorem}

\begin{proof}
 Let $(G,p)$ be a framework in $\ell_\infty^d$ with induced monochrome subgraphs $G_1,\ldots,G_d$.
 By \Cref{l:linfinitylaplacian}$(ii)$, the framework Laplacian  $L(G,p)$ is similar to the direct sum $\oplus_{i\in[d]} L(G_i)$. 
 Note that $\oplus_{i\in[d]} L(G_i)$ is a symmetric positive semidefinite $nd\times nd$ matrix.
 Also, for each $i\in[d]$ and each vertex $v\in V$, the diagonal $(v,v)$-entry of $L(G_i)$ is $\deg_{G_i}(v)$.
 Thus, by \Cref{l:upperbound},
 $$\lambda_{d+1}(L(G,p))  = \lambda_{d+1}\left(\oplus_{i\in[d]}L(G_i)\right) \leq \frac{n}{n-1}\,\min_{i\in[d]}\, \min_{v\in V}\,\deg_{G_i}(v).$$
 Note that,
$$\min_{i\in[d]}\, \min_{v\in V}\, \deg_{G_i}(v) 
\leq \biggl\lfloor\frac{1}{d} \sum_{i\in[d]} \min_{v\in V}\,\deg_{G_i}(v)\biggr\rfloor \leq 
\biggl\lfloor\frac{1}{d} \min_{v\in V}\,\deg_G(v)\biggr\rfloor.$$
The result follows.
\end{proof}

\begin{remark}
\Cref{thm:upperbound} is a $d$-dimensional generalisation of the following result due to Fiedler (\cite[\S3.5]{Fiedler73}): For any graph $G=(V,E)$ with $n$ vertices,
$$a(G)\leq \frac{n}{n-1}\,\min_{v\in V}\,\deg_G(v).$$
Fiedler's result corresponds to the $d=1$ case in the statement of \Cref{thm:upperbound}.

    An immediate consequence of \Cref{thm:upperbound} is that for any $d \geq 2$ and any $n \geq d+1$,  
\begin{equation*}
    a(K_n,\ell_\infty^d) \leq \frac{n}{n-1}  \biggl\lfloor\frac{n-1}{d} \biggr\rfloor \leq \frac{n}{d} = a(K_n)/d.
\end{equation*}
It follows that \Cref{thm:upperbound} gives a better upper bound for $a(K_n,\ell_\infty^d)$ than is provided by \Cref{c:linfbound}$(ii)$ if $n-1$ is not a multiple of $d$.

\Cref{thm:upperbound} also provides an analogue of the Alon-Boppana bound for regular graphs.
Specifically, if $G$ is a $k$-regular graph with $n$ vertices then, 
\begin{equation*}
    a(G,\ell_\infty^d) \leq \frac{n}{n-1}  \biggl\lfloor\frac{k}{d} \biggr\rfloor = (1+ o(1)) \biggl\lfloor\frac{k}{d} \biggr\rfloor.
\end{equation*}
In particular, if $G$ is a $2d$-regular graph then $a(G,\ell_\infty^d) \leq 2+o(1)$.
This latter upper bound will be improved upon in Section \ref{s:sparse}.
\end{remark}

\subsection{Calculations when \texorpdfstring{$d=2$}{d=2}}
Let $\mathcal{S}(G)$ denote the set of all spanning trees $T$ in a graph $G=(V,E)$ whose complement $G \setminus T$ is also a spanning tree in $G$.

\begin{proposition}\label{p:2dminrigidellinf}
    If a graph $G=(V,E)$ is a union of two edge-disjoint spanning trees then,
    \begin{equation*}
        a(G,\ell_\infty^2) = \max_{T\in\mathcal{S}(G)}\,\min \,\{ a(T),\, a(G\backslash T)\}.
    \end{equation*}
\end{proposition}

\begin{proof}
    By \Cref{t:l infinity decomp}, there exists a monochrome subgraph decomposition $(G_1,G_2)\in\mathcal{M}(G,\ell_\infty^2)$ such that $a(G,\ell_\infty^2)=\min\,\{a(G_1),\,a(G_2)\}$.
    If either $G_1$ or $G_2$ is not connected then, by \Cref{l:fiedler}$(i)$, $\min\,\{a(G_1),a(G_2)\}=0$. If $G_1$ and $G_2$ are both connected then they are both spanning trees since $G$ contains exactly $2(|V|-1)$ edges. In particular, the monochrome subgraph $G_1$ lies in $\mathcal{S}(G)$. 
    Thus,     $a(G,\ell_\infty^2) \leq \max_{T\in\mathcal{S}(G)}\,\min \,\{ a(T),\, a(G\backslash T)\}$.
    %The result now follows from \Cref{t:l infinity decomp}.

    For the reverse inequality, let $T\in\mathcal{S}(G)$. 
    By \cite[Theorem 4.3]{ck20}, there exists a framework $(G,p)$ in $\ell_\infty^2$ such that the induced monochrome subgraph decomposition for $(G,p)$ is the pair $(T,G\backslash T)$. Thus, by \Cref{t:l infinity decomp},
    $a(G,\ell_\infty^2) \geq \min\,\{ a(T),\, a(G\backslash T)\}$.
\end{proof}

\begin{proposition}
\label{p:K4}
    $a(K_4,\ell_\infty^2)=a(P_4)=2-\sqrt{2}$.
\end{proposition}

\begin{proof}
The complete graph $K_4$ is a union of two edge-disjoint spanning paths with 4 vertices.
Moreover, every spanning tree in $\mathcal{M}(K_4)$ is isomorphic to the path graph $P_4$.
Thus, by \Cref{p:2dminrigidellinf} and \Cref{ex:fiedler2}$(i)$,
$a(K_4,\ell_\infty^2) = a(P_4) = 2(1-\cos(\pi/4))$.
\end{proof}

In the proof of the following proposition, $K_3^{++}$ denotes the {\em bull graph} obtained by adjoining two degree one vertices to the complete graph $K_3$ such that the two new edges are non-adjacent. See the leftmost graph in \Cref{fig:list} for an illustration.

\begin{proposition}
\label{p:K5}
    $a(K_5,\ell_\infty^2)=a(K_3^{++})=\frac{1}{2}(5-\sqrt{13})$.
\end{proposition}

\begin{proof}
By \Cref{t:l infinity decomp}, $a(K_5,\ell_\infty^2) = \max_{(G_1,G_2) \in \mathcal{M}(K_5,\ell_\infty^2)}\, \min \{a(G_1),a(G_2)\}$.
Consider the framework $(K_5,p)$ in $\ell_\infty^2$ presented in \Cref{ex:K5}.
Note that the induced monochrome subgraphs $G_1$ and $G_2$ are both isomorphic to the bull graph $K_3^{++}$.
Thus, by  \Cref{t:l infinity decomp}, $a(K_5,\ell_\infty^2)\geq \min_{i=1,2} a(G_i) = a(K_3^{++}) = \frac{1}{2}(5-\sqrt{13})$
where the last equality follows by a direct calculation of the eigenvalues of the Laplacian matrix $L(K_3^{++})$.
To see that equality holds, note that every decomposition of $K_5$ into a pair of edge-disjoint connected spanning subgraphs will include one of the  graphs listed in \Cref{fig:list}.
For the second, third and fourth graphs in the list, a direct calculation shows that the algebraic connectivity is strictly less than that of the bull graph $K_3^{++}$.  
The fifth graph in the list is an odd cycle and so, by \Cref{t:oddhole}, this graph cannot arise in any monochrome subgraph decomposition of  $K_5$ in $\ell_\infty^2$. 
\end{proof}

\begin{figure}
    \centering
    \begin{tikzpicture}[scale=0.5]

% Coordinates of the vertices of K_3
\coordinate (v1) at (2,2);
\coordinate (v3) at (3,0);
\coordinate (v2) at (1,0);

% Coordinates of the new degree-1 vertices
\coordinate (v5) at (4,1);
\coordinate (v4) at (0,1);

% Draw the vertices of K_3
\fill (v1) circle (3pt) node[left] {$v_1$};
\fill (v2) circle (3pt) node[below] {$v_2$};
\fill (v3) circle (3pt) node[below] {$v_3$};

% Draw the new degree-1 vertices
\fill (v4) circle (3pt) node[above] {$v_4$};
\fill (v5) circle (3pt) node[above] {$v_5$};

% Draw the edges of K_3
\draw (v1) -- (v2);
\draw (v1) -- (v3);
\draw (v2) -- (v3);

% Add the edges for the new vertices (nonadjacent)
\draw (v2) -- (v4); % Edge from v1 to v4
\draw (v3) -- (v5); % Edge from v3 to v5

\end{tikzpicture}
       \begin{tikzpicture}[scale=0.5]

% Coordinates of the vertices
\coordinate (v1) at (0,0);
\coordinate (v2) at (2,0);
\coordinate (v3) at (4,0);
\coordinate (v4) at (4,2);
\coordinate (v5) at (2,2);

% Draw the vertices
\fill (v1) circle (3pt) node[below] {$v_1$};
\fill (v2) circle (3pt) node[below] {$v_2$};
\fill (v3) circle (3pt) node[below] {$v_3$};
\fill (v4) circle (3pt) node[right] {$v_4$};
\fill (v5) circle (3pt) node[left] {$v_5$};

% Draw the edges of the path graph
\draw (v1) -- (v2);
\draw (v2) -- (v3);
\draw (v3) -- (v4);
\draw (v2) -- (v5);
\draw (v4) -- (v5);

\end{tikzpicture}
    \begin{tikzpicture}[scale=0.5]

% Coordinates of the vertices
\coordinate (v1) at (0,2);
\coordinate (v2) at (0,0);
\coordinate (v3) at (2,0);
\coordinate (v4) at (4,0);
\coordinate (v5) at (4,2);

% Draw the vertices
\fill (v1) circle (3pt) node[left] {$v_1$};
\fill (v2) circle (3pt) node[below] {$v_2$};
\fill (v3) circle (3pt) node[below] {$v_3$};
\fill (v4) circle (3pt) node[below] {$v_4$};
\fill (v5) circle (3pt) node[right] {$v_5$};

% Draw the edges of the path graph
\draw (v1) -- (v2);
\draw (v2) -- (v3);
\draw (v3) -- (v4);
\draw (v4) -- (v5);

\end{tikzpicture}
\begin{tikzpicture}[scale=0.5]

% Coordinates of the vertices
\coordinate (v1) at (0,0);
\coordinate (v2) at (2,0);
\coordinate (v3) at (4,0);
\coordinate (v4) at (4,2);
\coordinate (v5) at (2,2);

% Draw the vertices
\fill (v1) circle (3pt) node[below] {$v_1$};
\fill (v2) circle (3pt) node[below] {$v_2$};
\fill (v3) circle (3pt) node[below] {$v_3$};
\fill (v4) circle (3pt) node[right] {$v_4$};
\fill (v5) circle (3pt) node[left] {$v_5$};

% Draw the edges of the path graph
\draw (v1) -- (v2);
\draw (v2) -- (v3);
\draw (v3) -- (v4);
%\draw (v1) -- (v5);
\draw (v2) -- (v5);

\end{tikzpicture}
\begin{tikzpicture}[scale=0.5]

% Coordinates of the vertices
\coordinate (v1) at (1.2,2);
\coordinate (v2) at (0.5,0.8);
\coordinate (v3) at (2,0);
\coordinate (v4) at (3.5,0.8);
\coordinate (v5) at (2.8,2);

% Draw the vertices
\fill (v1) circle (3pt) node[left] {$v_1$};
\fill (v2) circle (3pt) node[left] {$v_2$};
\fill (v3) circle (3pt) node[below] {$v_3$};
\fill (v4) circle (3pt) node[right] {$v_4$};
\fill (v5) circle (3pt) node[right] {$v_5$};

% Draw the edges of the path graph
\draw (v1) -- (v2);
\draw (v2) -- (v3);
\draw (v3) -- (v4);
\draw (v4) -- (v5);
\draw (v1) -- (v5);

\end{tikzpicture}

    \caption{List of graphs in the proof of \Cref{p:K5}.}
    \label{fig:list}
\end{figure}
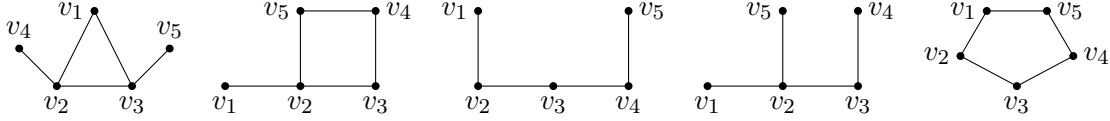

\subsection{Sparse graphs in \texorpdfstring{$\ell_\infty^d$}{l-infinity}}
\label{s:sparse}
If $G=(V,E)$ is a graph with $n:=|V|\geq d+1$ and at most $kn$ edges then, by \Cref{thm:upperbound} and the handshaking lemma,
\begin{equation*}
	a(G,\ell_\infty^d) \leq \frac{n}{n-1}  \biggl\lfloor\frac{2k}{d} \biggr\rfloor = (1+ o(1)) \biggl\lfloor\frac{2k}{d} \biggr\rfloor. 
\end{equation*}
In particular, if $k=d$ then $a(G,\ell_\infty^d) \leq 2 +o(1)$.
The following result improves on this latter bound.

\begin{theorem}
\label{t: sparse graphs}
    Let $G=(V,E)$ be a graph with at least $d+1$ vertices where  $d\geq 2$. If $|E| \leq d|V|$ then, 
    \begin{enumerate}[(i)]
    \item $a(G,\ell_\infty^d) \leq 1$.
    \item $a(G, \ell_\infty^d)=1$ if and only if $d=2$ and $G$ is the octahedral graph $K_{2,2,2}$.
\end{enumerate}
\end{theorem}

\begin{proof}
     By \Cref{t:l infinity decomp}, it suffices to show that for each monochrome subgraph decomposition $(G_1,\ldots,G_d) \in \mathcal{M}(G,\ell_\infty^d)$,
    we have $\min_{i \in [d]} a(G_i) \leq 1$ (with equality only if $d=2$ and $G=K_{2,2,2}$).

    Let $(G_1,\ldots,G_d) \in \mathcal{M}(G,\ell_\infty^d)$.
    If any $G_j$ is disconnected then, by \Cref{l:fiedler}$(i)$, 
    $a(G_j) = 0$ and so $\min_{i \in [d]} a(G_i) =0$.
    Suppose instead that each $G_i$ is connected. Note that, for each $j\in[d]$, there is no vertex in $G_j$ which is adjacent to every other vertex; indeed, if $v$ were such a vertex then $v$ would be an isolated vertex in every other monochrome subgraph $G_i$, $i \neq j$.
    Thus, if $G_j$ contains a cut vertex for some $j\in[d]$ then, by \Cref{l:kirkland}, $\min_{i \in [d]} a(G_i)\leq a(G_j)<1$.

    Now suppose that none of the monochrome subgraphs $G_1,\ldots,G_d$ contain a cut vertex. Since $|E|\leq d|V|$ it follows that each $G_i$ is a cycle of length $n=|V|$.
    Thus, by \Cref{ex:fiedler2}$(ii)$,
    \begin{equation}
    \label{eq:sparse}
        \min_{i \in [d]}\, a(G_i)  = a(C_n) = 2(1-\cos(2 \pi/n)).
    \end{equation}
    It follows that $\min_{i \in [d]} a(G_i) \leq 1$ if $|V| \geq 6$, with strict inequality if $|V|\geq 7$.

    Suppose further that $|V| \leq 6$.
    As $G$ is an edge-disjoint union of $d$ cycles, it follows that $d=2$, $|V| \in \{5,6\}$ and $G$ is 4-regular.
    The only two such graphs are the complete graph $K_5$ and the octahedron graph $K_{2,2,2}$. 
    By \Cref{p:K5}, $a(K_5,\ell_\infty^2)<1$.
    By the above, $a(K_{2,2,2},\ell_\infty^2)\leq 1$.
    It is shown in \cite[Example 7.10]{dewar2025uniquelyrealisablegraphspolyhedral} that there exists a framework $(K_{2,2,2},p)$ in $\ell_\infty^2$ such that each monochrome subgraph is a cycle of length 6.
    Thus,  by \Cref{t:l infinity decomp} and \Cref{ex:fiedler2}$(ii)$,
     \begin{equation*}
         a(K_{2,2,2}, \ell_\infty^2) \geq a(C_6) = 2(1-\cos(\pi/3)) =1.
     \end{equation*}
\end{proof}

We can make further improvements when $|E| <d|V|$ and the maximal degree is low.
For this, we require the following result of Kolokolnikov.

\begin{theorem}[{\cite[Theorem 1.2]{KOLO}}]\label{thm:KOLO}
    Let $T$ be a tree with $n$ vertices and maximal degree $\Delta$.
    Then
    \begin{equation*}
        a(T) \leq \frac{2(\Delta - 2)}{n} +  \frac{C_{\Delta} \log n}{n^2},
    \end{equation*}
    where the value of $C_{\Delta}>0$ is dependent only on $\Delta$.
\end{theorem}

\begin{corollary}\label{cor:maximaldegreelow}
    Let $G$ be a graph with at least $d+1$ vertices. If $|E| < d|V|$ then,
    \begin{equation*}
        a(G,\ell_\infty^d) \leq \frac{2(\Delta_G - d - 1)}{|V|} +  \frac{C \log |V|}{|V|^2},
    \end{equation*}
    where, given $C_{\Delta}$ is the constant described in \Cref{thm:KOLO}, $C = C_{\Delta_G-d +1}$.
\end{corollary}

\begin{proof}
    By \Cref{t:l infinity decomp},
    we can fix a decomposition $(G_1,\ldots,G_d)$ of $G$ where $a(G,\ell_\infty^d) = \min_{i \in [d]}a(G_i)$.
    We may restrict to the case where $a(G,\ell_\infty^d) > 0$,
    and hence each graph $G_i$ is connected.
     For each $G_i$, choose a spanning tree $T_i$.
     As $T_1,\ldots,T_d$ are edge-disjoint spanning trees in $G$,
     \begin{equation*}
         \left| E \setminus \bigcup_{i=1}^d E(T_i) \right| = |E| -\sum_{i=1}^d (|V| - 1) < d|V| - d(|V|-1) = d.
     \end{equation*}
     Hence there are at most $d-1$ edges not contained within one of the spanning trees.
	%Using the same argument as given in \Cref{t: sparse graphs},
	It follows that at least one of the monochrome subgraphs, $G_1$ say, is a tree.
	The maximal degree of $G_1$ is at most $\Delta_G - d +1$ since each graph $G_i$ must have positive minimal degree to be connected. 
    %monochromatic subgraph has minimal degree 1 or more).
    The result now follows from \Cref{thm:KOLO} applied to $G_1$.
\end{proof}

\subsection{Further calculations of \texorpdfstring{$a(K_n,\ell_\infty^d)$}{a(K,l-infinity)}}
The following are a selection of graphs whose algebraic connectivity in $\ell_\infty^d$ can be computed.

\begin{proposition}
\label{p:K6}
    $a(K_6,\ell_\infty^2)=1$.
\end{proposition}

\begin{proof}
    By \Cref{t: sparse graphs} and  \Cref{c:increasing}, $a(K_6,\ell_\infty^2)\geq a(K_{2,2,2},\ell_\infty^2)=1$.
    By \Cref{t:l infinity decomp}, there exists a monochrome subgraph decomposition $(G_1,G_2)\in\mathcal{M}(K_6,\ell_\infty^2)$ such that $a(K_6,\ell_\infty^2)=\min_{i=1,2} a(G_i)$. 
    By \Cref{l:fiedler}$(ii)$, if $G_i$ has vertex connectivity less than $2$ then $a(G_i)\leq 1$. 
    
    Suppose $G_1$ and $G_2$ both have vertex connectivity at least $2$.  
    If $G_i$ has at most 6 edges then it is a $6$-cycle, in which case $a(G_i) = 1$ by \Cref{ex:fiedler2}$(ii)$. 
    
    Suppose $|E(G_i)| \geq 7$ for $i=1,2$.
    Without loss of generality, assume that $|E(G_1)| = 7$ and $|E(G_2)|=8$ (since $|E(K_6)|=15$).
    As $G_1$ is 2-vertex-connected the degree sequence of $G_1$ must be $(2,2,2,2,3,3)$.
    In particular, $G_1$ cannot contain a degree $4$ vertex since this would be a cut vertex in $G_1$. Since $G_1$ does not contain a 5-hole (by \Cref{t:oddhole}), it must be isomorphic to the left hand graph given in \Cref{fig:K6}. 
    This implies $G_2$ is isomorphic to the right hand graph given in \Cref{fig:K6}.
    Direct calculation shows that $a(G_1)=a(G_2)=1$.
\end{proof}

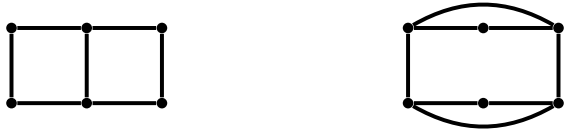
\begin{figure}[t]
\begin{center}
\begin{tikzpicture}[scale=1]
\begin{scope}
		\node[vertex] (1) at (0,0) {};
		\node[vertex] (2) at (1,0) {};
		\node[vertex] (3) at (1,1) {};
		\node[vertex] (4) at (0,1) {};
		\node[vertex] (x) at (2,0) {};
		\node[vertex] (y) at (2,1) {};
		\draw[edge] (1)edge(2);
		\draw[edge] (2)edge(3); 
		\draw[edge] (3)edge(4);  
		\draw[edge] (4)edge(1); 
        
		\draw[edge] (2)edge(x);
		\draw[edge] (x)edge(y); 
		\draw[edge] (y)edge(3); 
\end{scope}
\begin{scope}[xshift=150]
		\node[vertex] (1) at (0,0) {};
		\node[vertex] (2) at (1,0) {};
		\node[vertex] (3) at (1,1) {};
		\node[vertex] (4) at (0,1) {};
		\node[vertex] (x) at (2,0) {};
		\node[vertex] (y) at (2,1) {};
		\draw[edge] (1)edge(2);
		%\draw[edge] (2)edge(3); 
		\draw[edge] (3)edge(4);  
		\draw[edge] (4)edge(1); 
        
		\draw[edge] (2)edge(x);
		\draw[edge] (x)edge(y); 
		\draw[edge] (y)edge(3);

		\draw[edge,bend right] (1)edge(x);
		\draw[edge,bend left] (4)edge(y);
\end{scope}
	\end{tikzpicture}
	\end{center}
	\caption{The monochrome subgraphs $G_1$ (left) and $G_2$ (right) in the proof of \Cref{p:K6}.
    }
	\label{fig:K6}
\end{figure}

In the following, let $T_d$ be the unique tree with $2d$ vertices, diameter $3$ and two adjacent vertices, each with degree $d$ and adjacent to $d-1$ leaf vertices (see \Cref{fig:T3T4} for examples of $T_3$ and $T_4$).

\begin{lemma}
\label{l:K2dinddim}
Let $d\geq 2$.
\begin{enumerate}
\item There exists $p\in\mathcal{W}(K_{2d},\ell_\infty^d)$ such that every monochrome subgraph of the framework $(K_{2d},p)$ is isomorphic to $T_d$.
\item $a(K_{2d},\ell_\infty^d) \geq a(T_d)$.
\item There exists a spanning tree $T$ with maximum degree at most $d$ in the complete graph  $K_{2d}$  such that $a(K_{2d},\ell_\infty^d) = a(T)$.
\end{enumerate}
\end{lemma}

\begin{proof}
    A point $p\in\mathcal{W}(K_{2d},\ell_\infty^d)$ satisfying $(i)$ is constructed in the proof of \cite[Proposition 3.12]{dewar2025uniquelyrealisablegraphspolyhedral}.
    Statement $(ii)$ follows from $(i)$ and \Cref{t:l infinity decomp}.

   By \Cref{t:l infinity decomp}, there exists a monochrome subgraph decomposition $(G_1,\ldots,G_d)\in\mathcal{M}(K_{2d},\ell_\infty^d)$ such that $a(K_{2d},\ell_\infty^d)=\min_{i\in[d]}\, a(G_i)$.
    By $(ii)$ and Lemma \ref{l:fiedler}$(i)$, each monochrome subgraph $G_i$ is connected. Since the complete graph $K_{2d}$ has $d(2d-1)$ edges, it follows that each $G_i$ is a spanning tree in $K_6$.
    The maximum degree of each of these spanning  trees is at most $d$, as their degrees at each vertex must sum to $2d-1$ and every vertex has positive degree. This proves statement $(i)$.
\end{proof}

\begin{figure}[ht]
\begin{center}
\begin{tikzpicture}[scale=0.5]
		\node[vertex] (1) at (0,0) {};
		\node[vertex] (2) at (1,0) {};
		\node[vertex] (3) at (2,1) {};
		\node[vertex] (4) at (2,-1) {};
		\node[vertex] (5) at (-1,1) {};
		\node[vertex] (6) at (-1,-1) {};
		\draw[edge] (1)edge(2);
		\draw[edge] (1)edge(5);
		\draw[edge] (1)edge(6);
		\draw[edge] (2)edge(3);
		\draw[edge] (2)edge(4);
            \node (label) at (0.5,-2) {$a(T_3) \approx 0.438$.};
	\end{tikzpicture}\qquad\qquad\qquad
\begin{tikzpicture}[scale=0.6]
		\node[vertex] (1) at (0,0) {};
		\node[vertex] (2) at (1,0) {};
		\node[vertex] (3) at (2,1) {};
		\node[vertex] (4) at (2,-1) {};
		\node[vertex] (7) at (2.5,0) {};
		\node[vertex] (5) at (-1,1) {};
		\node[vertex] (6) at (-1,-1) {};
		\node[vertex] (8) at (-1.5,0) {};
		\draw[edge] (1)edge(2);
		\draw[edge] (1)edge(5);
		\draw[edge] (1)edge(6);
		\draw[edge] (1)edge(8);
		\draw[edge] (2)edge(3);
		\draw[edge] (2)edge(4);
		\draw[edge] (2)edge(7);
            \node (label) at (0.5,-2) {$a(T_4) \approx 0.354$.};
	\end{tikzpicture}
	\end{center}
	\caption{The graphs $T_3$ (left) and $T_4$ (right).}
	\label{fig:T3T4}
\end{figure}
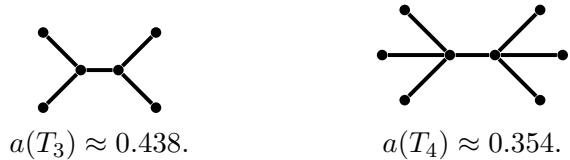

\begin{proposition}
\label{p:K63d}
    $a(K_6,\ell_\infty^3)= a(T_3) \approx 0.438$.    
\end{proposition}

\begin{proof}
    By \Cref{l:K2dinddim}$(ii)$, $a(K_6,\ell_\infty^3)\geq  a(T_3) \approx 0.438$.
    By \Cref{l:K2dinddim}$(iii)$, there exists a spanninng tree $T$ in $K_6$ with maximum degree at most $3$ such that
    $a(K_6,\ell_\infty^3)=a(T)$.
    There are only four such trees (see \Cref{fig:K63d}), and the one among them with the highest algebraic connectivity is $T_3$.
    Hence $a(K_6,\ell_\infty^3) \leq a(T_3)$, as required.
\end{proof}

\begin{figure}
    \centering
 \begin{tikzpicture}[scale=0.5]

% Coordinates of the vertices
\coordinate (v1) at (0,0);
\coordinate (v2) at (2,0);
\coordinate (v3) at (4,0);
\coordinate (v4) at (4,2);
\coordinate (v5) at (2,2);
\coordinate (v6) at (0,2);

% Draw the vertices
\fill (v1) circle (4pt) node[below] {};
\fill (v2) circle (4pt) node[below] {};
\fill (v3) circle (4pt) node[below] {};
\fill (v4) circle (4pt) node[right] {};
\fill (v5) circle (4pt) node[above] {};
\fill (v6) circle (4pt) node[left] {};

% Draw the edges of the path graph
\draw[edge] (v1) -- (v6) -- (v5) -- (v2) -- (v3) -- (v4);

\end{tikzpicture}
\quad
       \begin{tikzpicture}[scale=0.5]

% Coordinates of the vertices
\coordinate (v1) at (0,0);
\coordinate (v2) at (2,0);
\coordinate (v3) at (4,0);
\coordinate (v4) at (4,2);
\coordinate (v5) at (2,2);
\coordinate (v6) at (0,2);

% Draw the vertices
\fill (v1) circle (4pt) node[below] {};
\fill (v2) circle (4pt) node[below] {};
\fill (v3) circle (4pt) node[below] {};
\fill (v4) circle (4pt) node[right] {};
\fill (v5) circle (4pt) node[above] {};
\fill (v6) circle (4pt) node[left] {};

% Draw the edges of the path graph
\draw[edge] (v6) -- (v1) -- (v2) -- (v3) -- (v4);
\draw[edge] (v1) -- (v5);

\end{tikzpicture}
\quad
\begin{tikzpicture}[scale=0.5]
  % Coordinates of the vertices
\coordinate (v1) at (0,0);
\coordinate (v2) at (2,0);
\coordinate (v3) at (4,0);
\coordinate (v4) at (4,2);
\coordinate (v5) at (2,2);
\coordinate (v6) at (0,2);

% Draw the vertices
\fill (v1) circle (4pt) node[below] {};
\fill (v2) circle (4pt) node[below] {};
\fill (v3) circle (4pt) node[below] {};
\fill (v4) circle (4pt) node[right] {};
\fill (v5) circle (4pt) node[above] {};
\fill (v6) circle (4pt) node[left] {};

% Draw the edges of the path graph
\draw[edge] (v6) -- (v1) -- (v2) -- (v3) -- (v4);
\draw[edge] (v5) -- (v2);

\end{tikzpicture}
\quad
\begin{tikzpicture}[scale=0.5]

% Coordinates of the vertices
\coordinate (v1) at (0,0);
\coordinate (v2) at (1,1);
\coordinate (v3) at (4,0);
\coordinate (v4) at (4,2);
\coordinate (v5) at (3,1);
\coordinate (v6) at (0,2);

% Draw the vertices
\fill (v1) circle (4pt) node[below] {};
\fill (v2) circle (4pt) node[below] {};
\fill (v3) circle (4pt) node[below] {};
\fill (v4) circle (4pt) node[right] {};
\fill (v5) circle (4pt) node[above] {};
\fill (v6) circle (4pt) node[left] {};

% Draw the edges of the path graph
\draw[edge] (v1) -- (v2);
\draw[edge] (v5) -- (v3);
\draw[edge] (v2) -- (v5);
\draw[edge] (v2) -- (v6);
\draw[edge] (v4) -- (v5);

\end{tikzpicture}

    \caption{List of trees with $6$ vertices and maximum degree $3$ in the proof of \Cref{p:K63d}.}
    \label{fig:K63d}
\end{figure}
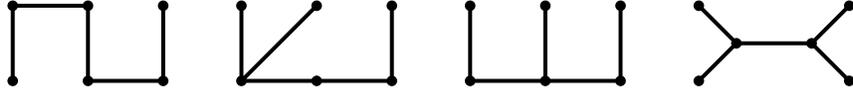

\begin{conjecture}
\label{p:K84d}
    $a(K_{2d},\ell_\infty^d)= a(T_d)$ for $d\geq 4$.
\end{conjecture}

\begin{remark}
\label{r:K84d}
    Note that, by \Cref{l:K2dinddim}$(ii)$, $a(K_8,\ell_\infty^4)\geq  a(T_4) \approx 0.354$.
    By \Cref{l:K2dinddim}$(iii)$, there exists a spanning tree $T$ in $K_8$ with maximum degree at most $4$ such that
    $a(K_8,\ell_\infty^4)=a(T)$.
    By \cite[Lemma 3.3]{YSZ08}, the tree $T$ must have  diameter at most $4$.
    There are $8$ such trees which are pictured in \Cref{fig:K84d} in increasing order with respect to their algebraic connectivities.

    Note that $T_4=H_7$ and so $H_8$ is the only tree in the list with an algebraic connectivity higher than that of $T_4$. Thus $T$ must be either $T_4$ or $H_8$.
    If $T=H_8$ then $K_8$ is an edge-disjoint union of four copies of $H_8$. Thus to establish the conjecture in the case $d=4$ it would be sufficient to show there is no monochrome subgraph decomposition in $\mathcal{M}(K_8,\ell_\infty^4)$ consisting of four copies of $H_8$.

    We also remark that the algebraic connectivity of $T_d$ is known to be the smallest root of the polynomial $p_d(x) := x^3 - (2d+2)x^2 + (d^2+2d+2)x - 2d$ (see \cite[Proposition 1]{Grone1990}).
\end{remark}

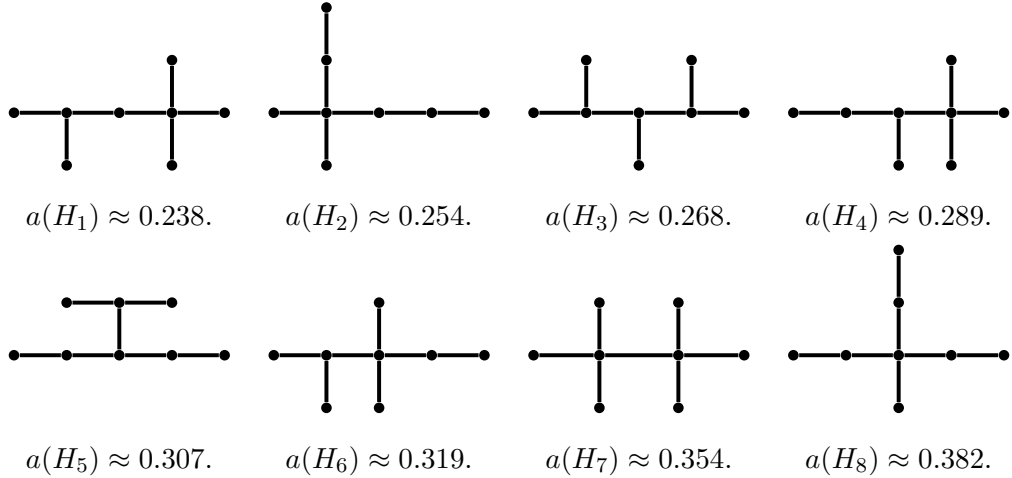
\begin{figure}
    \centering
\begin{tikzpicture}[scale=0.7]
		\node[vertex] (1) at (-2,0) {};
		\node[vertex] (2) at (-1,0) {};
		\node[vertex] (3) at (0,0) {};
		\node[vertex] (4) at (1,0) {};
		\node[vertex] (5) at (2,0) {};
		\node[vertex] (6) at (-1,-1) {};
		\node[vertex] (7) at (1,1) {};
		\node[vertex] (8) at (1,-1) {};
		\draw[edge] (1)edge(2);
		\draw[edge] (2)edge(3);
		\draw[edge] (3)edge(4);
		\draw[edge] (4)edge(5);
		\draw[edge] (6)edge(2);
		\draw[edge] (7)edge(4);
		\draw[edge] (8)edge(4);
            \node (label) at (0,-2) {$a(H_1) \approx 0.238$.};
	\end{tikzpicture}
\quad
\begin{tikzpicture}[scale=0.7]
		\node[vertex] (1) at (-2,0) {};
		\node[vertex] (2) at (-1,0) {};
		\node[vertex] (3) at (0,0) {};
		\node[vertex] (4) at (1,-0) {};
		\node[vertex] (5) at (2,0) {};
		\node[vertex] (6) at (-1,-1) {};
		\node[vertex] (7) at (-1,1) {};
		\node[vertex] (8) at (-1,2) {};
		\draw[edge] (1)edge(2);
		\draw[edge] (2)edge(3);
		\draw[edge] (3)edge(4);
		\draw[edge] (4)edge(5);
		\draw[edge] (6)edge(2);
		\draw[edge] (7)edge(2);
		\draw[edge] (7)edge(8);
            \node (label) at (0,-2) {$a(H_2) \approx 0.254$.};
	\end{tikzpicture}
\quad
\begin{tikzpicture}[scale=0.7]
		\node[vertex] (1) at (-2,0) {};
		\node[vertex] (2) at (-1,0) {};
		\node[vertex] (3) at (0,0) {};
		\node[vertex] (4) at (1,-0) {};
		\node[vertex] (5) at (2,0) {};
		\node[vertex] (6) at (0,-1) {};
		\node[vertex] (7) at (-1,1) {};
		\node[vertex] (8) at (1,1) {};
		\draw[edge] (1)edge(2);
		\draw[edge] (2)edge(3);
		\draw[edge] (3)edge(4);
		\draw[edge] (4)edge(5);
		\draw[edge] (6)edge(3);
		\draw[edge] (7)edge(2);
		\draw[edge] (8)edge(4);
            \node (label) at (0,-2) {$a(H_3) \approx 0.268$.};
	\end{tikzpicture}
\quad
\begin{tikzpicture}[scale=0.7]
		\node[vertex] (1) at (-2,0) {};
		\node[vertex] (2) at (-1,0) {};
		\node[vertex] (3) at (0,0) {};
		\node[vertex] (4) at (1,0) {};
		\node[vertex] (5) at (2,0) {};
		\node[vertex] (6) at (0,-1) {};
		\node[vertex] (7) at (1,1) {};
		\node[vertex] (8) at (1,-1) {};
		\draw[edge] (1)edge(2);
		\draw[edge] (2)edge(3);
		\draw[edge] (3)edge(4);
		\draw[edge] (4)edge(5);
		\draw[edge] (6)edge(3);
		\draw[edge] (7)edge(4);
		\draw[edge] (8)edge(4);
            \node (label) at (0,-2) {$a(H_4) \approx 0.289$.};
	\end{tikzpicture}
        \quad
  \begin{tikzpicture}[scale=0.7]
		\node[vertex] (1) at (-2,0) {};
		\node[vertex] (2) at (-1,0) {};
		\node[vertex] (3) at (0,0) {};
		\node[vertex] (4) at (1,0) {};
		\node[vertex] (5) at (2,0) {};
		\node[vertex] (6) at (1,1) {};
		\node[vertex] (7) at (0,1) {};
		\node[vertex] (8) at (-1,1) {};
		\draw[edge] (1)edge(2);
		\draw[edge] (2)edge(3);
		\draw[edge] (3)edge(4);
		\draw[edge] (4)edge(5);
		\draw[edge] (7)edge(6);
		\draw[edge] (7)edge(3);
		\draw[edge] (7)edge(8);
            \node (label) at (0,-2) {$a(H_5) \approx 0.307$.};
	\end{tikzpicture}
    \quad
    \begin{tikzpicture}[scale=0.7]
		\node[vertex] (1) at (-2,0) {};
		\node[vertex] (2) at (-1,0) {};
		\node[vertex] (3) at (0,0) {};
		\node[vertex] (4) at (1,0) {};
		\node[vertex] (5) at (2,0) {};
		\node[vertex] (6) at (-1,-1) {};
		\node[vertex] (7) at (0,1) {};
		\node[vertex] (8) at (0,-1) {};
		\draw[edge] (1)edge(2);
		\draw[edge] (2)edge(3);
		\draw[edge] (3)edge(4);
		\draw[edge] (4)edge(5);
		\draw[edge] (6)edge(2);
		\draw[edge] (7)edge(3);
		\draw[edge] (8)edge(3);
            \node (label) at (0,-2) {$a(H_6) \approx 0.319$.};
	\end{tikzpicture}
      \quad
    \begin{tikzpicture}[scale=0.7]
		\node[vertex] (1) at (-2,0) {};
		\node[vertex] (2) at (-0.75,0) {};
		\node[vertex] (3) at (0.75,1) {};
		\node[vertex] (4) at (0.75,0) {};
		\node[vertex] (5) at (2,0) {};
		\node[vertex] (6) at (-0.75,-1) {};
		\node[vertex] (7) at (-0.75,1) {};
		\node[vertex] (8) at (0.75,-1) {};
		\draw[edge] (2)edge(1);
		\draw[edge] (2)edge(6);
		\draw[edge] (2)edge(7);
		\draw[edge] (2)edge(4);
		\draw[edge] (4)edge(3);
		\draw[edge] (4)edge(5);
		\draw[edge] (4)edge(8);
            \node (label) at (0,-2) {$a(H_7) \approx 0.354$.};
	\end{tikzpicture}
\quad
  \begin{tikzpicture}[scale=0.7]
		\node[vertex] (1) at (-2,0) {};
		\node[vertex] (2) at (-1,0) {};
		\node[vertex] (3) at (0,0) {};
		\node[vertex] (4) at (1,0) {};
		\node[vertex] (5) at (2,0) {};
		\node[vertex] (6) at (0,-1) {};
		\node[vertex] (7) at (0,1) {};
		\node[vertex] (8) at (0,2) {};
		\draw[edge] (1)edge(2);
		\draw[edge] (2)edge(3);
		\draw[edge] (3)edge(4);
		\draw[edge] (4)edge(5);
		\draw[edge] (6)edge(3);
		\draw[edge] (7)edge(3);
		\draw[edge] (7)edge(8);
            \node (label) at (0,-2) {$a(H_8) \approx 0.382$.};
	\end{tikzpicture}

    \caption{List of trees with $8$ vertices, maximum degree at most $4$ and diameter at most $4$ in  \Cref{r:K84d}. Note that $T_4=H_7$.}
    \label{fig:K84d}
\end{figure}

\section{Redundant rigidity}
\label{s:rr}
A framework $(G,p)$ in a normed space $X$ is said to be {\em vertex-redundantly  rigid}  if it is infinitesimally rigid and every framework $(H,p_H)$ obtained by deleting a vertex $v_0$ from $G$ together with its incident edges, and setting $p_H=(p_v)_{v\in V(H)}$, is infinitesimally rigid.
A graph $G=(V,E)$ is {\em vertex-redundantly rigid in $X$} if there exists a vertex-redundantly rigid framework $(G,p)$ in $X$.

\begin{proposition}\label{p:rr}
    Let $X=(\mathbb{R}^d,\|\cdot\|_X)$ and let $G=(V,E)$ be a graph with at least $d+2$ vertices.
     \begin{enumerate}
    \item If $a(G,X) > \gamma(X)$ then  $G$ is vertex-redundantly rigid in $X$.
    \item   If $G$ is minimally rigid in $X$ then  $G$ is not vertex-redundantly rigid in $X$.
    \end{enumerate}
\end{proposition}

\begin{proof}
     $(i)$: Let $U'$ be the set of points $p\in (\mathbb{R}^d)^V$ such that, for each vertex $v_0$ in $G$, the set $\{p_v:v\in V\backslash\{v_0\}\}$ has full affine span in $X$. Note that $U'$ is an open and dense subset of $(\mathbb{R}^d)^V$. Thus, by \Cref{p:dense}$(ii)$, there exists $p\in \mathcal{W}(G,X)\cap U'$  such that  $\lambda_{k(X)+1}(L(G,p)) >\gamma(X)$. 
     By \Cref{p:vertexdeletion} and \Cref{l:rigidityeigenvalue}$(ii)$,  the framework $(G,p)$ is vertex-redundantly rigid. 

     $(ii)$: If $G$ is vertex-redundantly rigid in $X$ then $G-v$ is rigid in $X$ for all $v\in V$.
    By \cite[Corollary 4.13]{Dewar21},  $|E(G-v)| \geq d(|V|-1) - k(X)$ for each $v \in V$, and $|E|=d|V|-k(X)$.
    Hence, for any $v \in V$,
    \begin{equation*}
        d|V|-k(X) - \deg_G(v) = |E| - \deg_G(v) = |E(G-v)| \geq d(|V|-1) - k(X),
    \end{equation*}
    which implies $\deg_G(v) \leq d$ for each $v \in V$.
    Since $|V| \geq d+2$ and $k(X) \leq \binom{d+1}{2}$,
    \begin{equation*}
        |E| \leq \frac{d}{2}|V| < d|V| - \tbinom{d+1}{2} \leq d|V| - k(X),
    \end{equation*}
    contradicting that $|E|=d|V|-k(X)$.
\end{proof}

\begin{proposition}
    \label{p:mr}
    Let $X=(\mathbb{R}^d,\|\cdot\|_X)$ and let $G=(V,E)$ be a graph with at least $d+1$ vertices. If $G$ is minimally rigid in $X$ then
     $a(G,X) \leq \gamma(X)$.
\end{proposition}

\begin{proof}
    First suppose $|V| \geq d+2$.
    By \Cref{p:rr}$(ii)$, $G$ is not vertex-redundantly rigid in $X$. Thus, the result follows by \Cref{p:rr}$(i)$.
    
    Next suppose $|V|=d+1$.
    As $G$ is minimally rigid in $X$,
    it follows from \cite[Theorem 5.8]{Dewar21} that $X$ is isometrically isomorphic to $d$-dimensional Euclidean space. Let $\Psi: \ell_2^d \rightarrow X$ be a linear isometry.
    By \cite[Theorem 1.2]{LNPR23}, $a(G,\ell_2^d) = 1$.
    Thus, by \Cref{c:isometric-alg},
     $   a(G,X) \leq \lambda_n(\Psi^\top\Psi)\,a(G,\ell_2^d)  = \|\Psi \|_2^2$.
    Note that $\|\cdot\|_X^* = \|\cdot\|_X$ and so,
    $$\|\Psi\|_2=\sup_{\|x\|_2=1}\|\Psi(x)\|_2 = \sup_{\|\Psi(x)\|_X=1}\|\Psi(x)\|_2 = \sup_{\|y\|_X^*=1}\|y\|_2 = \gamma(X)^{\tfrac{1}{2}}.$$
    Hence $a(G,X) \leq \gamma(X)$.
\end{proof}

A framework $(G,p)$ in $X$ is said to be {\em edge-redundantly  rigid}  if it is infinitesimally rigid and every framework obtained by deleting an edge $vw$ from $G$ is infinitesimally rigid.
A graph $G$ is {\em edge-redundantly rigid} in  $X$ if there exists a framework $(G,p)$ in $X$ which is edge-redundantly rigid.

\begin{proposition}
    \label{p:k2d+1}
	For every $d \geq 2$,
	the complete graph $K_{2d+1}$ is not edge-redundantly rigid in $\ell_\infty^d$. 
\end{proposition}

\begin{proof}
	Suppose for contradiction that there exists an edge-redundantly rigid framework $(K_{2d+1},p)$ in $\ell_\infty^d$.
	Each monochrome subgraph of $(K_{2d+1},p)$ is 2-edge-connected,
	and hence must have at least $2d+1$ edges.
	In fact, as $K_{2d+1}$ has $d(2d+1)$ edges,
	each monochrome subgraph has exactly $2d+1$ edges.
	Hence each monochrome subgraph of $(G,p)$ is a spanning cycle.
    However, this contradicts \Cref{t:oddhole}.
\end{proof}

It follows from \Cref{p:k2d+1} that $2d+2$ or more vertices are needed for edge-redundant rigidity in $\ell_\infty^d$.
Because of this, the authors would (somewhat intrepidly) conjecture the following.

\begin{conjecture}\label{conj:redrig}
	For every $d \geq 2$ and every $n \geq 2d+2$,
	the complete graph $K_{n}$ is edge-redundantly rigid in $\ell_\infty^d$. 
\end{conjecture}

\begin{remark}
\Cref{conj:redrig} is true when $d=2$. To see this, first observe that for $n = 6$ we can take $p \in \mathcal{W}(K_6,\ell_\infty^2)$ as described for the octahedral graph $K_{2,2,2}$ in \cite[Example 7.10]{dewar2025uniquelyrealisablegraphspolyhedral} and obtain an edge-redundantly rigid framework $(K_6,p)$.
For $n=7$,  take the previous framework $(K_6,p)$ and add the new vertex at $(0.5,0.9)$, and for higher values of $n$ we can add additional vertices at generic points sufficiently close to $(0.5,0.9)$. 
\end{remark}

\subsection*{Acknowledgements}
SD was supported by the Heilbronn Institute for Mathematical Research.
DK was supported by a Mary Immaculate College Research Sabbatical Award.

\bibliographystyle{plainurl}
\bibliography{ref}

\end{document}